\documentclass[a4paper,11pt]{amsart}
\usepackage{amsmath,amsfonts,amssymb,latexsym}
\usepackage{times}
\newtheorem{theorem}{\bf \mbox{Theorem}}[section]
\newtheorem{proposition}[theorem]{\bf \mbox{Proposition}}
\newtheorem{lemma}[theorem]{\bf \mbox{Lemma}}
\newtheorem{definition}[theorem]{\bf \mbox{Definition}}
\newtheorem{corollary}[theorem]{\bf \mbox{Corollary}}

\newtheorem{remark}{\bf Remark}
\newcommand{\field}[1]{\mathbb{#1}}
\newcommand{\C}{\field{C}}

\newcommand{\N}{\field{N}}
\newcommand{\Q}{\field{Q}}
\newcommand{\R}{\field{R}}
\newcommand{\Cal}[1]{\mathcal{#1}}
\newcommand{\wt}[1]{\widetilde{#1}}
\newcommand{\adj}{{\rm adj}}

\newcommand{\brN}{{\overline{\N}}}
\newcommand{\bg}{{\bf g}}
\newcommand{\bG}{{\bf G}}
\newcommand{\bh}{{\bf h}}
\newcommand{\bm}{{\bf m}}
\newcommand{\bp}{{\bf p}}
\newcommand{\bP}{{\bf P}}
\newcommand{\bq}{{\bf q}}

\newcommand{\bu}{{\bf u}}
\newcommand{\bU}{{\bf U}}
\newcommand{\bru}{\bar{u}}
\newcommand{\bv}{{\bf v}}
\newcommand{\bV}{{\bf V}}
\newcommand{\brv}{\bar{v}}
\newcommand{\cB}{\Cal{B}}
\newcommand{\cC}{\Cal{C}}
\newcommand{\cD}{\Cal{D}}
\newcommand{\cE}{\Cal{E}}
\newcommand{\cF}{{\Cal{F}}}
\newcommand{\cI}{\Cal{I}}
\newcommand{\cJ}{\Cal{J}}
\newcommand{\cK}{\Cal{K}}
\newcommand{\cL}{\Cal{L}}
\newcommand{\cM}{\Cal{M}}
\newcommand{\cN}{\Cal{N}}
\newcommand{\cO}{\Cal{O}}
\newcommand{\cosupp}{{\rm co}$-${\rm supp}}
\newcommand{\cQ}{\Cal{Q}}
\newcommand{\cS}{\Cal{S}}
\newcommand{\cU}{\Cal{U}}
\newcommand{\cV}{\Cal{V}}
\newcommand{\cX}{\Cal{X}}
\newcommand{\dd}{{\partial}}

\newcommand{\Dlt}{{\Delta}}
\newcommand{\gm}{{{\gamma}}}
\newcommand{\Gm}{{\Gamma}}
\newcommand{\kp}{{\kappa}}
\newcommand{\la}{{\langle}}
\newcommand{\omg}{{\omega}}
\newcommand{\Omg}{{\Omega}}
\newcommand{\ra}{{\rangle}}
\newcommand{\rd}{{\rm d}}
\newcommand{\rdlg}{{{\rm d}}_{\log}}
\newcommand{\reg}{{\rm reg}}
\newcommand{\sing}{{\rm sing}}
\newcommand{\sgm}{\sigma}
\newcommand{\Sgm}{\Sigma}
\newcommand{\str}{{\rm str}}
\newcommand{\te}{{\theta}}
\newcommand{\Te}{{\Theta}}
\newcommand{\tns}{{\otimes}}
\newcommand{\ula}{{\underline{a}}} 
\newcommand{\ulb}{{\underline{b}}}
\newcommand{\ve}{\varepsilon}
\newcommand{\vp}{\varphi}
\newcommand{\wik}{{\rm weak}}
\newcommand{\wta}{{\wt{\ula}}}
\newcommand{\wtb}{{\wt{\ulb}}}
\newcommand{\wtbt}{{\wt{\beta}}}

\newcommand{\wtbp}{{\wt{\bp}}}
\newcommand{\wtcD}{\wt{\cD}}

\newcommand{\wtcU}{{\wt{\cU}}}
\newcommand{\wtcS}{{\wt{\cS}}}

\newcommand{\wtE}{\wt{E}}
\newcommand{\wtbG}{\wt{\bG}}

\newcommand{\wtM}{\wt{M}}
\newcommand{\wtpi}{{\wt{\pi}}}
\newcommand{\wtS}{{\wt{S}}}
\newcommand{\wtsgm}{\wt{\sgm}}
\newcommand{\wtZ}{{\wt{Z}}}
\newcommand{\wtX}{{\wt{X}}}

\newcommand{\bo}{{\bf 0}}

\begin{document}
\title{Monomialization of singular metrics on real surfaces}
\author{Vincent Grandjean}
\address{Departamento de Matem\'atica, Universidade Federal do Cear\'a
(UFC), Campus do Pici, Bloco 914, Cep. 60455-760. Fortaleza-Ce,
Brasil.}
\email{\tt vgrandjean@mat.ufc.br}
\dedicatory{To celebrate Felipe Cano on his sixtieth birthday.}
\thanks{I am very pleased to thank D. Grieser, E. Bierstone and P. Milman
for their time, support, and what I learned from them, allowing me
to write this paper. I am most grateful for P. Milman's attention and unshakable patience towards me. 
Very big thanks to both referees for very helpful suggestions, comments and remarks.}
\begin{abstract}
We present a result of monomialization of regular $2$-symmetric tensors over regular real surfaces.
For the inner metric of an embedded real surface singularity, the result obtained 
for the regular extension of the pull-back of the inner metric over a resolved surface is more precise
for applications than a Hsiang \& Pati type resolution of singularities.
\end{abstract}
\maketitle
\tableofcontents
\section{Introduction and statement of a simpler version of the result} 
Any smooth $2$-symmetric tensor over a smooth Riemannian manifold 
$M$ is point-wise diagonalizable. The collection of these diagonalizing bases rarely forms a 
smooth local orthogonal frame everywhere on $M$.

\smallskip
Any real analytic family $(A(u))_{u\in \cU}$ of $n\times n$-symmetric matrices over an open subset $\cU$ 
of an Euclidean space $\R^p$, can be re-parameterized by a real analytic proper surjective mapping $\pi:\cV \to \cU$,
a finite composition of geometrically admissible blowings-up, which 
locally simultaneously diagonalizes everywhere on $\cV$ 
the pulled-back family $(A\circ\pi(v))_{v\in\cV}$, see Kurdyka \& Paunescu \cite{KuPa} and our generalization  \cite{Gra2}
(with a different proof).

\smallskip 
Our primary motivation is two fold: locally describe, as simply as can be
(i) real analytic $2$-symmetric tensors over a real analytic Riemannian manifold which may degenerate
somewhere, (ii) the restriction of real analytic $2$-symmetric tensors over a real analytic 
Riemannian (or Hermitian) ambient manifold $M$ to (the regular part of) a given real (or complex) analytic 
singular sub-variety $S$. 
\\
The model example to have in mind is the \em inner metric \em on $S$, that is the restriction of the ambient metric
to the regular part of the analytic subset $S$.
The problem at hands in this model example is understanding how the inner Riemannian structure on the 
regular part of the singular sub-variety accumulates at the singular part.  
An accepted scheme to treat this problem, which hopefully will yield a local description of the inner 
metric nearby the singular locus, is to re-parameterize the (embedded) singular sub-variety as a regular 
manifold by a (bi-rational like) regular surjective mapping (the resolution mapping of 
an embedded desingularization of the singular sub-variety). The pull back of the 
inner metric onto the resolved manifold then becomes a regular semi-Riemannian metric, losing
positive-definiteness along the pre-image of the singular locus of the original sub-variety 
(the exceptional locus of the resolution). 
Further \em carefully chosen \em blowings-up should yield a better description 
of the iterated pull-back of the metric. The meaning of \em carefully chosen \em 
is yet to be systematically developed in regard of the control that can be 
guaranteed after pull-back.

\smallskip
Hsiang \& Pati were first to provide a local description of inner metric of 
normal complex surface singularity germs \cite[Sections II \& III]{HP} along the proposed scheme. 
Later Pardon \& Stern conceptualized this point of view \cite[Section 3]{PaSt}. 
Grieser also found a real analytic version (unpublished) of Hsiang \& Pati's result \cite{Gri1}.
Almost fifteen years after Hsiang \& Pati, Youssin announced the existence of
such a local description 
on some resolved manifold of any given complex algebraic singularity \cite{Yo}. 
The short version can be stated as follows:

\smallskip\noindent
{\bf Youssin Conjecture \cite{Yo,BBGM}:}  
\em Given a singular complex algebraic sub-variety $X_0$ of pure dimension $n$ embedded in a complex 
(algebraic) manifold $M_0$ , 
there exists a resolution of singularities $\sgm:(X,E) \to X_0$ of $X_0$, which is a locally finite composition 
of blowings-up with regular (algebraic) centers, such that at each point $\ula$ of 
the exceptional divisor $E$, the pull-back, onto the resolved manifold $X$, of the inner metric of $X_0$ 
(inherited from the ambient Hermitian metric in $M_0$) by the resolution mapping $\sgm$ is (locally) quasi-isometric 
to a sum of the "Hermitian squares" of (exactly $n$ independent) differentials of monomials in 
the exceptional divisor $E$. Furthermore, the integer vectors made of the power of each monomials must be
linearly independent and totally ordered. \em

\smallskip\noindent
Our joint article \cite{BBGM} 
proposes a problem equivalent to Youssin's, which is better adapted to resolution singularities
techniques, proving Youssin's Conjecture in the case of real and complex 
(algebraic or analytic) surfaces and $3$-folds singularities.  

\medskip
The present paper focuses only on real analytic surface singularities
from the above point of view. Each of our joint works,  
\cite{GrGr} deals with the geodesics of cuspidal surface singularities ending at the singular point,
while \cite{GrSa} and \cite{Gra1} deal with germs of gradient differential equations over real surface
singularities, requires a tailored description of a singular metric over a regular surface
nearby an exceptional curve, finer than Hsiang \& Pati's normal form. 

Given a real analytic surface singularity $S$ in an ambient real analytic manifold $M$ endowed with a 
real analytic metric, given a $2$-symmetric tensor $\kp$ on $M$, can we find, as in the case 
of the inner metric, 
a useful and simple local presentation of the pull-back, on a 
resolved surface, of the restriction of $\kp$ to the regular part of $S$?

\medskip
Considered from the point of view of a resolved manifold, the
problems presented above are 
occurrences of the following general situation:
Let $X_1$ be a regular (i.e. real or complex algebraic or analytic) manifold 
and let $(B_1,\bg_1)$ be a regular vector bundle over $X_1$ of finite rank and equipped with a 
regular fiber-metric (Riemannian or Hermitian) $\bg_1$.
Let $\kp_1$ be a regular $2$-symmetric tensor field on $B_1$, (later shortened as 
$2$-symmetric tensor on $B_1$). 
Some features of $\kp_1$ can be investigated following two similar but different approaches: 

\smallskip\noindent
{\bf Parameterization Problem.} \em Find a regular and surjective mapping $\sgm_2: X_2 \to X_1$, 
and describe the features of $\kp_1\circ\sgm_2$ as objects with regular variations on $X_2$. 
Note that $\kp_1\circ\sgm_2$ is a $2$-symmetric tensor on the regular vector bundle $\sgm_2^*B_1$. \em

\smallskip\noindent
{\bf Resolution Problem.} 
\em When there exists a regular mapping of vector bundles $TX_1 \to B_1$ which is an isomorphism
outside a proper sub-variety of $X_1$, find a regular and surjective mapping $\sgm_2: X_2 \to X_1$ such that  
the features of the pulled-back $\kp_2$ of the $2$-symmetric tensor $\kp_1$, now a
$2$-symmetric tensor on $X_2$, are as good as can be. \em

\smallskip\noindent
The expression "as good as can be" is vague, indeed. But, depending on motivations, the end product of a 
solution of the resolution problem may vary. An instance of "as good as can be" is obtaining Hsiang \& Pati's 
local form.

\smallskip\noindent
{\bf Resolution Problem refined.} 
\em The pulled-back tensor $\kp_2$ can locally be written as a weighted sum of symmetric squares of two $1$-forms 
such that, at each point of a (explicit) simple normal crossing divisor $D$ containing the exceptional divisor, 
the push-forward of the kernels of these two $1$-forms by the resolution mapping are orthogonal. Moreover these 
two $1$-forms admit only simple singularities (if any), and the corresponding coefficients (weights) in the diagonal 
presentation of $\kp_2$ are monomial in $D$ times a unit 
\em

\smallskip
The Parameterization Problem consists mostly of
regularizing some functions (roots, components of vector fields, minors,...). 
Kurdyka \& Paunescu's results, generalizing Rellich's real and Kato's complex one dimensional case,
provide almost immediately an answer to the Parameterization Problem when 
$X_1$ is an open subset of an Euclidean space. Our generalization \cite{Gra2} provides the full answer.

The Resolution Problem, in the described refinement form of Hsiang \& Pati, has never been 
addressed. It is significantly harder than the associated Parameterization Problem, although
in practice we very likely have to start with solving this latter one. 
 
\bigskip
The present paper proposes an answer to the Resolution Problem refined for a "singular" $2$-symmetric tensor of 
a real analytic vector bundle of rank two over a regular real analytic surface,
isomorphic to the tangent bundle of the surface outside a sub-variety of positive codimension. 
This applies after preliminary preparations to the 
inner metric of a real analytic surface singularity embedded in real analytic Riemannian manifold. 
Below, our results are stated for the inner 
metric of a real analytic singular surface, the simplest possible situation.

\smallskip
Let $X_0$ be a real analytic surface singularity of a real analytic Riemannian manifold $(M_0,\bg_0)$,
supporting a real analytic space structure $(X_0,\cO_{X_0} := \cO_{M_0}/I_0)$, for a coherent 
$\cO_{M_0}$-ideal sheaf $I_0$.
Let $\gm_0: T(X_0\setminus Y_0) \to \bG(2,TM_0|_{X_0\setminus Y_0})$ be the Gauss
mapping of $X_0$, for $Y_0$ the (non-empty) singular locus of $X_0$. 
\\
Our first result (see Proposition \ref{prop:param-reg} for the general case) is a parameterization 
result of a global simultaneous diagonalization, namely:
\begin{proposition}\label{prop:param-simpler}
There exists $\sgm_1: (X_1,E_1) \to (X_0,Y_0)$, a locally finite composition of geometrically admissible 
blowings-up, such that $X_1$ is a regular surface, the Gauss mapping $\gm_0\circ\sgm_1$ extends as
a real analytic mapping $\gm_1: X_1 \to \bG(2,\sgm_1^* TM_0|X_1)$ over $X_1$ and $E_1$ is a simple normal
crossing divisor. 
Let $B_1$ be the real analytic vector sub-bundle of rank $2$ of $\sgm_1^*TM|_{X_1}$ corresponding to 
$\gm_1$, extending $\sgm_1^*T(X_0\setminus Y_0)$ over $X_1$. 

\smallskip
There exists 
an $\cO_{X_1}$-invertible sheaf $\cQ_1$ of sections of $S^2B_1^\vee$
with empty co-support such that at each point $\ula_1$ of $X_1$ there exists 
an open neighbourhood $\cU_1$ of $\ula_1$ with the following properties:

0) there exists, up to permutation, a pair (unique if the inner metric is not constant 
along the fibers) $\te_1$, $\te_2$ of regular sections of $B_1^\vee|_{\cU_1}$ both nowhere 
vanishing in $\cU_1$ such that $\cQ_1|_{\cU_1}$ is generated by $\te_1\cdot\te_2$, and
 
i) for each $\ula \in \cU_1$, the vector lines  $\ker \theta_1 (\ula)$ and $\ker \theta_2 (\ula)$ of 
$T_{\sgm_1(\ula)} M_0$ intersect orthogonally; 

ii) If $\sgm_1 (\ula)$ is not a singular point of $X_0$, then $\ker \te_i (\ula)$ is contained in 
$T_{\sgm_1(\ula)} X_0$;

iii) The metric parameterized by $\sgm_1$ writes on $\cU_1$ as 
\begin{center}
$(\bg_0\circ\sgm_1|_{B_1})|_{\cU_1} = (\alpha_1 \cdot \te_1)\tns(\alpha_1 \cdot \te_1) + 
(\alpha_2 \cdot \te_2)\tns (\alpha_2 \cdot \te_2) = \alpha_1^2\te_1\tns\te_1 + \alpha_2^2\te_2\tns\te_2$; 
\end{center}
where $\alpha_i$ is a unit for $i=1,2$.
\end{proposition}
In other words, a global regular object, the quadratic ideal $\cQ_1$, locally induces an orthogonal 
directional frame on the regular vector bundle $B_1$ over $X_1$, locally simultaneously diagonalizing
$(\bg_0|_{X_0})\circ\sgm_1$, so that the size of the 
(local) generators of the diagonalizing frame are monomials (here local units - compare with \cite{KuPa}).
This form is stable under further points blowings-up. Then, we get the \em desingularization \em 
of the pulled-back metric, leading to Theorem \ref{thm:main} presented in the following simpler form: 
\begin{theorem}\label{thm:main-simpler}
There exists a locally finite composition of geometrically admissible blowings-up 
\begin{center}
$\sgm_2: (X_2,E_2) {\buildrel\beta_2\over\longrightarrow} (X_1,E_1) {\buildrel\sgm_1\over\longrightarrow} (X_0,Y_0)$ 
\end{center}
such that each point $\ula_2$ of $X_2$ admits an open neighbourhood $\cU_2$ 
and, up to permutation, a pair (unique if not constant along the fiber)
$\omg_1$, $\omg_2$ of  regular sections of $TX_2^\vee|_{\cU_2}$ such that

i) Each foliation $\cD_i$ over $\cU_2$ generated by $\omg_i$, $i=1,2$, admits only simple singularities adapted to $E_2$;
  
ii) If $\sgm_2 (\ula_2)$ is a regular point of $X_0$, the lines  $D\sgm_2(\ula_2)\cdot \ker \omg_i (\ula_2)$, for $i=1,2$
are orthogonal in $T_{\sgm_2(\ula_2)} X_0$;

iii) The pull-back of the metric $\bg_0|_{X_0}$ by the resolution mapping $\sgm_2$ 
extends on $X_2$ as a real analytic semi-Riemannian metric 
$\bg_2$, which on $\cU_2$ is written as
\begin{center}
$\bg_2= (\cM_1 \cdot \omg_1)\tns(\cM_1 \cdot \omg_1) + 
(\cM_2 \cdot \omg_2)\tns (\cM_2 \cdot \omg_2) = \cM_1^2\omg_1\tns\omg_1 + \cM_2^2\omg_2\tns \omg_2$
\end{center}
where $\cM_i$ is a monomial in $E_2$, and $i=1,2$.
\end{theorem}
Without further blowings-up, we obtain the following 
(cf. Corollary \ref{cor:real-HP-smooth} and Corollary \ref{cor:real-HP-corner}):

\smallskip\noindent
{\bf Corollary.}
\em Each point of $E_2$ admits Hsiang \& Pati coordinates: 
\\
i) For $\ula_2$ a regular point of $E_2$, there are local coordinates $(u,v)$ at $\ula_2$ 
such that $(E_2,\ula_2) =\{u=0\}$ and the extension of the pulled-back metric $\bg_2$ is quasi-isometric 
nearby $\ula_2$ to
\begin{center}
$\rd u^{k+1} \tns \rd u^{k+1}  + \rd (u^{l+1}v) \tns \rd (u^{l+1}v)$ 
\end{center}
for non-negative integer numbers $l\geq k$.
\\
ii) For $\ula_2$ a corner point of $E_2$, there are local coordinates $(u,v)$ at $\ula_2$ such 
that $(E_2,\ula_2) =\{uv=0\}$ and the extension of the pulled-back metric $\bg_2$ is quasi-isometric 
nearby $\ula_2$ to 
\begin{center}
$\rd (u^mv^n)\tns \rd (u^mv^n)  + \rd (u^kv^l)\tns \rd (u^kv^l)$ 
\end{center}
for positive integer numbers $m\leq k, n\leq l$ and such that $ml - kn \neq 0$.
\em

\medskip
This article is organized as follows:
Section \ref{section:not-voc} presents basic material needed throughout the
paper and set some notations. 
Theorem  \ref{thm:desing} states Hironaka's resolution of singularities
\cite{Hi1,AHV} in a form best suited for our purpose.
Section \ref{section:quadratic} introduces further classical definitions.
The parameterization problem strictly speaking is solved in Section \ref{section:param-quad}
 by Proposition \ref{prop:param-reg} 
(given above in a special case).
Section \ref{section:res-foliat} recalls what is a resolution of singularities of
a real analytic plane singular foliation. 
Section \ref{section:pairs-foliat} investigates pairs of generically transverse
plane singular foliations, gives some technical results 
about the behavior of such a pair with respect to a prescribed normal crossing divisor. 
We recall the logarithmic point of view of the reduction of singularities of a 
plane foliation \cite{Ca1,Ca2,CCD}. 
What is developed in this section is essential to obtain 
the announced normal form.
Our main result, Theorem \ref{thm:main} presented in 
Section \ref{section:main}, solves the resolution problem in the refined version presented above.
The local normal form of the pull-back of the initial 
$2$-symmetric tensor is stable under further blowings-up.
Subsection \ref{subsection:appendixA1} and Subsection \ref{subsection:appendixA2}
deal with the notion of the restriction of $2$-symmetric tensor onto a singular sub-variety. 
To that end, we show Proposition \ref{prop:gauss-regular}, 
stating the existence of Gauss-regular resolution of singularities (see also \cite{BBGM}).
Appendix \ref{section:normal-form} describes the unexpected consequence of our detailed 
Theorem \ref{thm:main} and its completion Corollary \ref{cor:main}, 
for inner metrics on singular surfaces.
Proposition \ref{prop:metric-smooth} gives a normal form of the pull-back of the ambient 
metric on the resolved surface at any regular point of the exceptional divisor.

But for Hironaka's resolution of singularities and reduction of singularities of plane vector fields germs,
the paper is self contained, at the low cost of reproving already known results (or their variations).
%
%
%
%
%
%
%
%
%
%
%
%
%
%
%
%
%
%
%
%
%
%
%
%
%
%
%
%
%
%
%
%
%
%
%
%
%
%
%
%
%
\section{Setting - Resolution of Singularities Theorems}\label{section:not-voc}
\smallskip
In what follows the adjective \em analytic \em only means \em real analytic. \em

\medskip\noindent
A \em regular manifold \em is an analytic manifold. 
A \em regular sub-manifold \em of a given regular manifold is an analytic sub-manifold.
A \em regular mapping \em $M \to N$ is an analytic mapping between regular manifolds $M$ and $N$.
A \em sub-variety \em is a closed real analytic subset of a given regular manifold. 
A \em regular sub-variety \em is a sub-variety and a regular sub-manifold. 
A point $\ula$ of a given irreducible sub-variety $X$ is \em regular \em if the 
germ $(X,\ula)$ is a regular sub-manifold germ of dimension $\dim X$.
A point of a sub-variety is \em regular \em if it is a regular point of the irreducible 
component containing the given point.
A point of a sub-variety is \em singular \em if it is not regular. The singular locus of 
an irreducible sub-variety is a sub-variety of dimension strictly lower, if not empty.
Let $\cO_M$ be the sheaf of analytic function germs on the regular manifold $M$.
We will speak of \em $\cO_M$-ideals and $\cO_M$-modules \em to mean $\cO_M$-ideal sheaves and $\cO_M$-module
sheaves.

\medskip
A \em normal crossing divisor of a regular manifold $M$ of dimension $n$ \em is the co-support $D$
of a principal $\cO_M$-ideal of finite type which is locally monomial at each point of $M$.  
It is called a \em simple normal crossing divisor \em if furthermore each of its irreducible component is 
regular. We will shorten as \em (s)nc-divisor. \em 

Let $D$ be a normal crossing divisor of $M$. Each point $\ula \in M$ admits local regular coordinates 
$(\bu,\bv)=(u_1,\ldots,u_s;\bv)$ centered at $\ula$, with $0\leq s\leq n$, such that the germ 
of $D$ at $\ula$ writes $(D,\ula) =\{u_1\cdots u_s =0\}$. Such local coordinates 
are said \em adapted to the (simple) normal crossing divisor $D$ at $\ula$. \em  

\smallskip
Let $\ula$ be a point of $M$ and let $\cO_\ula := \cO_{M,\ula}$ be the regular local ring of the germ $(M,\ula)$.
Let $\bm_\ula$ be its maximal ideal and let $n = \dim (M,\ula) := \dim_\R \cO_\ula/\bm_\ula$. 
\\
A \em local monomial $\cM$ (at $\ula$) in the nc-divisor $(D,\ula)$ \em
is a function germ of $\cO_\ula$ such that there exist local regular coordinates $(\bu,\bv)=(u_1,\ldots,u_s;\bv)$ 
adapted to $D$ at $\ula$ in which 
\begin{center}
$\cM = \pm \Pi_{i=1}^s u_i^{p_i}$, for non-negative integer numbers $p_1,\ldots, p_s$. 
\end{center}
A principal $\cO_M$-ideal of finite type $I$ is \em monomial in the nc-divisor $D$ \em
if at each point of $M$ there exists local coordinates $(\bu,\bv) = (u_1,\ldots,u_s;\bv)$ 
adapted to $D$ at $\ula$ such that the local generator at $\ula$ of the ideal $I$ is a local
monomial in the nc-divisor $D$. 
 
\smallskip
Two local monomials $\cM_1 = \Pi_{i=1}^s u_i^{p_i}$ and $\cM_2 = \Pi_{i=1}^s u_i^{q_i}$ in the nc-divisor 
$(D,\ula)$ are said \em ordered \em if, either $p_i \geq q_i$ for each $i$ or $p_i\leq q_i$ for each $i$.
Any finite set of local monomials in $(D,\ula)$ is \em well ordered \em if any pair of distinct monomials is ordered.
Any finite set of monomials in the nc-divisor is well ordered if it is well ordered. 

\smallskip
Let $D$ be a nc-divisor of $M$. A regular sub-manifold $C$ of $M$ 
is \em normal crossing with $D$ at $\ula$ \em if up to a change of coordinates 
adapted to $D$ at $\ula$, we have 
$$
(C,\ula) = \{u_1 = \cdots = u_r = v_1 = \cdots = v_t = 0\}
$$ 
for $0\leq r\leq s$ and $0\leq t\leq n-s$.

\medskip
Let $(M,D)$ be a regular manifold with a nc-divisor $D$ (possibly empty).
A blowing-up with center a regular sub-variety $C$ is \em geometrically admissible \em if it
is normal crossing with $D$ at each point (see \cite[p. 213]{BM2} for a more restrictive definition). 
Assume the center $C$ is of codimension greater than or equal to 
$2$, and let $\beta_C: (M',E') \to (M,D)$ be such a blowing-up, then $E' := \beta_C^{-1}(D\cup C)$ 
is a nc-divisor, and is a snc-divisor if $D$ were. 

Let $Z$ be any subset of $M$. The \em strict transform of \em $Z$ by $\beta_C$ is 
defined as the analytic Zariski closure of $\beta_C^{-1} (Z\setminus C)$
and is denoted $Z^\str$. When a sub-variety $Z$ can be equipped with a real analytic space 
structure, the strict transform admits an algebraic description (see below). 
If $\gm: (M'',E'') \to (M',E')$ is a locally finite 
sequence of geometrically admissible blowings-up, we write again
$Z^\str$ for the strict transform of $Z$ by $\beta_C\circ\gm$.
Suppose furthermore that the (s)nc-divisor $D$ is the exceptional divisor 
of a locally finite sequence of geometrically admissible blowings-up 
$\pi: (M,D) \to N$, for some regular manifold $N$. We find $E' = D^\str \cup \beta_C^{-1} (C)$.  
Nevertheless {\bf strict transforms of an existing exceptional divisor 
will be denoted by the same letter as the exceptional divisor}:
$E' = D \cup \beta_C^{-1} (C)$.

\smallskip
We will use, almost systematically, the resolution of  
singularities of Hironaka, in the following embedded real setting.
\begin{theorem}[Embedded resolution of singularities \cite{Hi1,AHV,BM2}]\label{thm:desing}
Let $M$ be a (connected) regular manifold.

\smallskip
1) Let $I$ be a (non-zero and coherent) $\cO_M$-ideal sheaf. There exists a locally finite composition 
of geometrically admissible blowings-up $\pi: (\wtM,E_{\wtM})  \to M$ such that the total transform $\pi^* I$ 
is a principal ideal, monomial in the snc-divisor $E_{\wtM}$.

\smallskip
2) Let $X$ be a sub-variety of $M$, of positive codimension, for which there 
exists a coherent $\cO_M$-ideal sheaf with co-support $X$.
Let $Y$ be the singular locus of $X$.
There exists $\pi: (\wtM,\wtX,E_{\wtM}) \to (M,X,Y)$, a locally finite composition of 
geometrically admissible blowings-up, such that $\wtX :=\pi^{-1}(X\setminus Y)^\str$ is 
a regular sub-variety of $\wtM$, normal crossing with the snc-divisor $E_{\wtM}:= \pi^{-1}(Y)$ 
and such that $\wtX \cap E_{\wtM}$ is a snc-divisor of $\wtX$.   
\end{theorem}
Although real algebraic sub-varieties can always be equipped with a ringed space structure 
induced by a coherent ideal sheaf, in order to be desingularized, real analytic singular sub-varieties 
(see \cite[Sect. 4]{GMT}) must also be equipped with a real analytic space structure, which, 
unlike their complex counter-part, is not always possible. (See \cite[Sect. 10]{BM2} for a 
proper account on the category of ringed spaces that can be desingularized). 
Therefore we ask for the following condition to be satisfied:

\smallskip\noindent
{\bf Hypothesis.} Any singular sub-variety to be desingularized  
admits a coherent ideal sheaf of the structural sheaf of the ambient regular manifold with co-support 
the given sub-variety, so that the sub-variety is equipped with the corresponding real analytic 
space structure. 

\smallskip
Let $X$ be a sub-variety of a regular analytic manifold $M$, equipped with a real analytic 
space structure given by a coherent $\cO_M$-ideal sheaf $I$. 
Let $\beta_C: (M',E') \to (M,C)$ be a geometrically admissible blowing-up
with center $C$ which does not contain all irreducible components of $X$.
The coherent $\cO_{M'}$-ideal $\beta_C^* I$ gives rise to $I^\str$, a coherent $\cO_{M'}$-ideal obtain
locally as follows: factors out of any local section of $\beta_C^* I$ any term that vanishes along the germ of $E'$.
The co-support of this ideal is the strict transform $X^\str$ of $X$ we mentioned earlier
(see \cite[Section 3, p. 237]{BM2}). In contrast the weak transform of $X$ is obtained as follows: the ideal 
$\beta_C^* I$ factors out as $I_{E'}^m\cdot I'$, where $I_{E'}$ is the coherent $\cO_{M'}$-ideal sheaf of 
function germs vanishing along $E'$ and $m$ is a non-negative integer number and $I'$ is a coherent $\cO_{M'}$-ideal 
sheaf with co-support the \em weak transform \em $X^\wik$ of $X$, which does not contain $E'$. 
The weak transform $X^\wik$ always contains the strict transform $X^\str$, and they may differ because of their respective 
intersection with the exceptional divisor $E'$. Note that if $I$ is principal, then both notions coincide.

\smallskip
We recall and will use the following result of ordering any finite family of Monomials.
\begin{theorem}[Ordering Monomials \cite{BM2}]\label{thm:order-monom}
Let $M$ be a (connected) regular manifold. Let $D$ be a (s)nc-divisor such that 
each of its component is regular. Let $I_1,\ldots,I_k$ be principal $\cO_M$-ideals monomial 
in the nc-divisor $D$. 
There exists $\pi: N\to M$, a locally finite sequence of geometrically admissible blowings-up 
(with centers normal crossing with $D$ and its iterated total transforms), such that
the pulled-back ideals $\pi^*I_1,\cdots,\pi^*I_k$ are principal $\cO_N$-ideals monomial in 
the (s)nc-divisor $\pi^{-1} (D)$ such that at each point of $N$ the local generators of 
$\pi^*I_1,\cdots,\pi^*I_k$ are ordered.
\end{theorem}

\medskip\noindent
{\bf Some notations.} 

\smallskip
We will use $Unit$ to mean any regular function germ which is a local unit and for which 
a more specific notation is not necessary.

We will write $const$ to mean a non-zero constant we do not want to precise further. 

We will write $(...)$ to mean a regular function germ we do not want to 
denote specifically. 

If $z$ is a component of some local coordinates system centered at some given point 
then $z^{+\infty}$ means the null function germ.

Let $\brN : = \N \cup \{+\infty\}$ and $\brN_{\geq t} := \{n\in \brN: n \geq t\}$.  
%
%
%
%
%
%
%
%
%
%
%
%
%
%
%
%
%
%
%
%
%
%
\section{Bilinear symmetric forms}\label{section:quadratic}
Let $V$ be a real vector space of finite dimension. 

The real vector space tensor product $V\tns V$ decomposes as the direct sum of the real vector subspaces 
$S^2(V) \oplus \wedge^2V$, where $S^2(V)$ is the $2$-nd symmetric power of $V$ and 
$\wedge^2V$ is the $2$-nd exterior power of $V$.
A symmetric bilinear form of $V$ is just an element of $S(2,V):= S^2(V^\vee)$ for $V^\vee$ the dual vector space of $V$. 

Let $Q(V)$ be the real vector space of real quadratic forms on $V$. 
There is a canonical isomorphism $Q(V) \to S(2,V)$ mapping a quadratic form $\kp$ 
onto its polar form $\kp_\Delta$. 

\medskip
Let $M$ be a connected regular manifold of finite dimension. 
Let $F$ be a regular vector bundle of finite rank over $M$. 
Let $F^\vee$ be the dual bundle and $\bP F$ be the projective bundle associated with $F$.
Let $\Gm_M (F)$ be the $\cO_M$-module of regular sections of $F$. 
Let $S(2,F):=S^2(F^\vee)$, respectively $Q(F)$, be the vector bundle over $M$ of $2$-symmetric tensors,
respectively quadratic forms, over the fibers of $F$. 
These two vector bundles are canonically isomorphic via a regular mapping of vector bundles over $M$.
A regular \em quadratic form on $F$ \em is a regular section $M\to Q(F)$ and a 
\em $2$-symmetric regular tensor field on $F$ \em is a regular section $M \to S(2,F)$, shortened 
as \em $2$-symmetric tensor on $F$. \em
When $F = TM$ we just say \em $2$-symmetric tensor on $M$, \em 
respectively \em quadratic form on $M$. \em

\smallskip
A (regular) \em fiber-metric \em on $F$ is a regular section $M\to S(2,F)$ such that the 
corresponding quadratic form is everywhere positive definite. Such a regular fiber-metric 
always exists by Grauert and Morrey Theorem \cite{Grau,Mor}.
When $F$ is equipped with a (regular) fiber metric $\bg$ we write $|-|_\bg$ for the norm 
and $\langle -, -\rangle_\bg$ for the corresponding scalar product 

\medskip
As for foliations (see Section \ref{section:res-foliat}), it is more convenient 
to study invertible sub-modules of $2$-symmetric tensors than sections. 

\smallskip

An invertible $\cO_M$-sub-module $\cL$ of $\Gm_M (S(2,F))$
\em vanishes at the point $\ula$ of $M$, \em if the bilinear symmetric form 
$\kp(\ula)$ over $T_\ula M)$ is the null form, for $\kp$ is a local generator of $\cL$ nearby $\ula$. 
The \em vanishing locus $V(\cL)$ of $\cL$ \em is the co-support of $\cL$ and is a sub-variety of $M$. 
Since $M$ is connected, we say that $\cL$ is \em non-zero \em (or \em non-null\em) if $V(\cL)$ 
is of positive codimension. We say that $\cL$ \em 
does not vanish over a subset $S$ \em if $V(\cL) \cap S$ is empty.
%
%
%
%
%
%
%
%
%
%
%
%
%
%
%
%
%
%
%
%
%
%
%
%
%
%
%
%
%
%
%
%
%
%
%
%
\section{Good parameterization of $2$-symmetric tensors on regular surfaces}\label{section:param-quad}
Proposition \ref{prop:param-reg}, the main result of this section, is the first important step 
towards our main result. It is essentially a precise local simultaneous diagonalization result for a 
global orthogonal (directional) frame. It is a consequence of Theorem \ref{thm:param-reg} (\cite{KuPa}
for the local case and our paper \cite{Gra2} for the global case).

\medskip
\noindent
Since there will be competitive notations further down in the paper, we do take the following 

\medskip\noindent
{\bf Notations for pull-backs:} \em Let $B$ be a smooth vector bundle over a smooth manifold $N$. Let $\sgm : M\to N$ be a smooth mapping.
The pulled-back vector bundle $\sgm^* B$ of $B$ on $M$ along the mapping $\sgm$ is denoted from now on by $B^\sgm$.

If $A$ is a sub-module or sub-module sheaf over $N$ 
of sections of $B$ or $B^\vee$ or $\wedge^2 B$ or $S(2,B)$ (more generally of sections of 
tensors over $B$), we will denote $A^\sgm$ for $A\circ\sgm$. 

Let $\kp$ be any section which is tensorial in/over the fibers of $B$. We will denote the pull-back 
$\kp\circ\sgm$, tensorial in/over the fibers of $B^\sgm$, by $\kp^\sgm$.

If $I$ is an ideal or ideal sheaf over $N$ the pull-back of this ideal by
$\sgm$ is still denoted in the standard way: $\sgm^*I = I\circ\sgm$. 
\em

\medskip
We recall the following
\begin{theorem}[\cite{KuPa,Gra2}]\label{thm:param-reg}
Let $M$ be a regular connected manifold, and let $B$ be a regular vector bundle over $M$ of rank $r\geq 1$ 
equipped with a regular fiber-metric $\bg$. Let $\kp : M \to S(2,B)$ be a regular symmetric operator
respectively to the fiber-metric $\bg$. 

There exists $\sgm:M' \to M$ a locally finite composition of geometrically admissible blowings-up such that 
every point $\ula'$ of $M'$ admits neighbourhood $\cU'$ in $M'$ and a regular 
frame $(\xi_1,\ldots,\xi_r)$ of $(B^\sgm)|_{\cU'}$ orthonormal for $\bg^\sgm|_{\cU'}$ in which $\kp^\sgm(\ulb')$ 
is diagonal at every point $\ulb'$ of $\cU'$.
\end{theorem}
Truth be told, the conclusion of Theorem \ref{thm:param-reg} still holds for a coherent invertible 
$\cO_S$-sheaf of sections of $S(2,B)$ by our method \cite{Gra2}. 

\medskip
In Section \ref{section:not-voc}, we introduced the notions of strict transform and weak transform.
Let $S$ be a regular surface and let $I$ be a coherent
$\cO_S$-ideal with co-support $X$ of dimension $1$. If $\sgm :(S',E') \to S$ is any sequence of points blowings-up,
the difference between the strict transform $X^\str$ and its weak transform $X^\wik$ may consist of an isolated subset of the 
exceptional divisor $E'$. We blow-up to principalize and monomialize some ideals, 
making the support of their total transform a snc-divisor. At this stage weak transform and total 
transform of the co-support of the initial ideal are equal. It is a  snc-divisor which is normal crossing with
the exceptional divisor of the principalization and monomialization process. Note that strict transform 
and weak transform of the co-support of a principal ideal are equal.

\medskip
Let $S$ be a connected regular surface and let $B$ be a regular vector bundle of rank two over $S$ equipped
with a regular fiber-metric $\bg$.

Suppose given $\cL$ a non-zero invertible  $\cO_S$-sub-module of sections of $S(2,B)$.
The \em ideal $\cC_\cL$ of coefficients of $\cL$ \em is the coherent $\cO_S$-ideal obtained by the evaluation of 
$\cL$ (by means of local generators) along the (local) regular sections $S\to B\times_S B$. 
The \em vanishing locus \em $V(\cL)$ is the co-support 
of $\cC_\cL$. The ideal $\cC_\cL$ is of finite type by definition, thus $\cO_S$-coherent. 

\smallskip
The \em degeneracy locus \em $D(\cL)$ of $\cL$ is the sub-variety of $S$ consisting of the points at which 
any local 
generator $\kp$ gives rise to a degenerate bilinear symmetric form on the corresponding fiber, locus which  
contains $V(\cL)$. If $D(\cL)$ is empty, we say that $\cL$ is \em non-degenerate. \em
If $D(\cL)$ is whole of $S$, we say that $\cL$ is \em everywhere degenerate. \em
If $D(\cL)$ is everywhere of positive local codimension, we say that $\cL$ is \em generically non-degenerate. \em

\smallskip
Given local (and trivializing) coordinates $(u,v;X,Y)$ of $B$ at $\ula$ in $S$, we can write $\ula$ a local 
generator of $\cL$ at $\ula$ as 
\begin{center}
$(\kp (\ula)) ((X_1,Y_1),(X_2,Y_2)) = a X_1 X_2 + b (X_1 Y_2 + Y_1 X_2) + cY_1 Y_2$,
\end{center}
The $\cO_S$-ideal locally generated at (any) $\ula$ by $ac -b^2$ is also of finite type, 
whose co-support is exactly the degeneracy locus $D(\cL)$, and where $a,b,c$ lie in $\cO_{S,\ula}$.

\smallskip\noindent
{\bf Notation.} \em We denote this latter $\cO_S$-ideal by $I_\cL^D$. \em

\medskip
\smallskip
Let $\sgm: R \to S$ be a locally finite composition of point blowings-up and let $E$ be the exceptional 
divisor. By definition, the pull-back $\cL^\sgm$ of $\cL$ by $\sgm$ is the invertible $\cO_R$-module 
of $\Gm_R(S(2,B^\sgm))$ locally generated by $\kp^\sgm$ (observe that $S(2,B)^\sgm = S(2,B^\sgm)$).

\smallskip
Since $\cL$ is non-zero, the total transform of $\cC_\cL$ factors out as 
$\sgm^*\cC_\cL = J\cdot K$, for a principal $\cO_S$-ideal $J$ monomial in $E$ and  
an $\cO_S$-ideal $K$ whose co-support does not contain any component of $E$. 
The invertible $\cO_R$-submodule of $\Gm_R (S(2,B^\sgm))$ defined as 
\begin{center}
$\cL^{\sgm,div}:=(\cL^\sgm)^{div} := J^{-1}\cdot(\cL^\sgm)$
\end{center}
is called the \em divided pull-back \em 
of $\cL$ by $\sgm$. If $\kp'$ is a local generator of
$\cL^{\sgm,div}$ then the form $\kp'$ is the null bilinear symmetric form at $\ula$ of $R$ if and only if 
$\ula$ lies in  $\cosupp (K)$. 

\smallskip 
As we did for $\cL$, we also define $D(\cL^{\sgm,div})$, the \em degeneracy locus \em of 
$\cL^{\sgm,div}$, that is the set of points $\ula$ of $R$ where the bilinear symmetric form
$\kp' (\ula)$ is degenerate, for $\kp'$ a local generator of $\cL^{\sgm,div}$ at $\ula$.  
If $\cL$ is non-degenerate (respectively everywhere-degenerate, respectively generically 
degenerate), so is $\cL^{\sgm,div}$.

When $\cL$ is generically non-degenerate, we find out that the ideal $\sgm^*I_\cL^D$ decomposes 
as $\sgm^*I_\cL^D = J^D\cdot K^D$, where $J^D$ is a principal $\cO_R$-ideal monomial in 
$E$, while $K^D$ is a $\cO_R$-ideal whose co-support does not contain any component of $E$. 
Since $I_\cL^D \subset \cC_\cL^2$, the ideal $J^D$ is contained in the ideal $J^2$, and thus 
deduces that 
\begin{center}
$D(\cL^{\sgm,div}) = \cosupp (J^{-2} J^D K^D)$. 
\end{center}

\smallskip
We need to prepare a little bit our situation with 
the following
\begin{lemma}\label{lem:param-1}
There exists $\sgm_1:(S_1,E_1) \to S$, a locally finite sequence of points 
blowings-up, so that the total transform $\sgm_1^*\cC_\cL$ is a principal and 
monomial $\cO_{S_1}$-ideal in the snc-divisor $V_\cL:= \sgm_1^{-1} (V(\cL))$, 
which contains the exceptional divisor $E_1$. Furthermore, if $\cL$ is generically non-degenerate, 
the total transform $\sgm_1^*I_\cL^D$ is principal and monomial in the snc-divisor $D_\cL:= (\sgm_1)^{-1} (D(\cL))$ 
which is normal crossing with the snc-divisor $E_1\cup V_\cL^\str$.
\end{lemma}
\begin{proof}
It is straightforward from principalization and monomialization of ideals \cite{Hi1,BM1}, as quoted in point 1) of 
Theorem \ref{thm:desing}.
\end{proof}
%
%
We need some further preparatory material.
We will work mostly locally, with germs. These local data will be gathered
in an appropriate module or ideal sheaf.

\smallskip
Let $B_1 := B^{\sgm_1}$ and let $\cL_1$ be the invertible 
$\cO_{S_1}$-module 
\begin{center}
$\cL_1 := \cL^{\sgm_1,div}$
\end{center} 
(by Lemma \ref{lem:param-1}) and let $\kp_1$ be a local generator of $\cL_1$ 
over some open neighbourhood $\cU_1$ (of some given point). The $2$-tensor $\kp_1$ vanishes nowhere in $\cU_1$. 
Let $\bg_1$ be the fiber-metric $\bg^\sgm_1$. 
Let  $(u,v;X,Y)$ be regular trivializing coordinates of $B_1$ over $\cU_1$.
Let $G_1(\ula)$ be the matrix of $\bg_1(\ula)$ and $K_1(\ula)$ that of $\kp_1$ for any $\ula\in \cU_1$.
The mappings $\ula\to G_1(\ula),K_1(\ula)$ are regular over $\cU_1$. Up to shrinking $\cU_1$ and
applying Gram-Schmidt orthonormalization to the frame over $\cU_1$, we can assume that
$G_1$ is the identity matrix at every point of $\cU_1$.
Thus for any $\bU,\bV$ sections of $B_1|_{\cU_1}$ we get
$$
\kp_1(\bU,\bV) = \bg_1(K_1\bU,\bV) = \la K_1\bU,\bV\ra .
$$
At any point $\ula =(u,v)$ of $\cU_1$ we can write 
$$
\kp_1(\ula)((X_1,Y_1),(X_2,Y_2)) = a_1 X_1X_2 + b_1(X_1Y_2 + X_2Y_1) + c_1 Y_1 Y_2 .
$$
By Lemma \ref{lem:param-1} we find that $a_1c_1 - b_1^2 = \cM_1\cdot \vp_1$, where $\cM_1$ is a monomial 
and $\vp_1$ is a unit over $\cU_1$, or $a_1c_1 - b_1^2$ is identically null, while the ideal generated by 
$(a_1,b_1,c_1)$ is $\cO_{S_1}|_{\cU_1}$.
Following the idea of \cite{Gra2}, the quadratic form in $B_1|_{\cU_1}$ 
$$
K_1\bU \wedge \bU = b_1(Y^2-X^2) +(a_1-c_1)XY =: P_1(\ula)(X,Y) 
$$
vanishes either along two distinct lines or along the full plane $B_{1,\ula}$.
Note that
$$
\{ \bU \in B_{1,\ula} :  K_1\bU \wedge \bU = 0 \} = \{\bU \in B_{1,\ula} :  \rd_\bU \kp_1(\ula) 
\wedge \rd_\bU \bg_1(\ula) = 0\},
$$
therefore is independent on the choice of local coordinates on $B_1$.
Let $Q_1(\ula)$ be the line of $S(2,B_{1,\ula})$ generated by the quadratic polynomial $P_1(\ula)$. 
This vector line depends only on $\cL_1$ at $\ula$ and that is why we can glue all the 
local $Q_1(\ula)$ together to construct a $\cO_{S_1}$-coherent invertible sheaf $\cQ_{\cL_1}$ of $\Gm_{S_1}(S(2,B_1))$.
%
%
%
%

\medskip
Let $VD_{\cL_1}$ be the \em vertical discriminant locus of $\cL_1$, \em that is the set of points $\ula$ in $S_1$ 
at which the principal quadratic ideal $Q_1(\ula)$ is null. 
Let $I_{\cL_1}^{VD}$ be the $\cO_{S_1}$-ideal sheaf locally generated by the coefficients of the 
homogeneous quadratic polynomial $P_1$.
Therefore the co-support of $I_{\cL_1}^{VD}$ is exactly $VD_{\cL_1}$.

\smallskip
The invertible module $\cL$ is said \em constant along the fiber \em if its vertical 
critical locus is $B$.
\begin{lemma}\label{lem:param-2}
- If $\cL$ is not constant along the fibers, 
there exists a locally finite sequence of points blowings-up $\beta_2:(S_2,E_2) \to (S_1,E_1)$ 
so that the total transform $\beta_2^*I_{\cL_1}^{VD}$ is a principal $\cO_{S_2}$-ideal sheaf which is
monomial in the vertical discriminant, the snc-divisor,
\begin{center}
$VD_\cL := \beta_2^{-1} (\cosupp (I_{\cL_1}^{VD}))$
\end{center}
which is normal crossing with the snc-divisor $E_2\cup V_\cL^\str\cup D_\cL^\str$.

- If $\cL$ is constant along the fibers then we define $S_2 :=S_1$, $E_2 := E_1$ and $\beta_2$ 
is the identity map of $S_1$.
\end{lemma}
\begin{proof}
It is again straightforward from principalization and monomialization of ideals.
\end{proof}

\medskip
Let $\cX$ be a $\cO_S$-invertible sheaf of regular sections of $B$. Assume that 
the co-support $\sing(\cX)$ of $\cX$ has codimension $2$ if not empty.
Since $B$ admits a fibered structure, the data of such a $\cX$ is equivalent to have
$\Te := \cX^\vee$ the $\cO_S$-invertible sheaf of sections of $B^\vee$ dual to $\cX$.
Although we may not be able to define an invertible sheaf $\cX^\perp$ of $B$, the \em orthogonal of
$\cX$, \em we can find an invertible sheaf of sections of $\bP B|_{S\setminus \sing(\cX)}$ 
orthogonal to that generated by $\cX$, say $\cX^{\perp,\bP}$, the orthogonal direction to $\cX$. 
After a locally finite composition of points blowings-up $\gm: (S',E') \to S$
we know that $\cX^\gm = \cJ'\cdot\cX'$ where $\cJ'$ is an $\cO_{S'}$-ideal
principal an monomial in $E'$ and $\cX'$ is an invertible $\cO_{S'}$-sub-modules of regular sections of $B^\gm$
with empty co-support. The orthogonal direction invertible sheaf of $\cX'$ is thus well defined over $S'$.

\medskip
We now can state the main result of this section 
\begin{proposition}\label{prop:param-reg}
Let $\cL$ be a non-zero invertible $\cO_S$-module of $\Gm_S (S(2,B))$. 
There exists a locally finite sequence of points blowings-up $\sgm_R: (R,E_R) \to S$ such that,

\medskip\noindent
1) Assuming that $\cL$ is not constant along the fibers. 

\smallskip
i) The conclusions of Lemma \ref{lem:param-1} and 
Lemma \ref{lem:param-2} hold true for $\sgm_R : =\sgm_1\circ\beta_2$. 

\smallskip
ii) Let $B_R := B^{\sgm_R}$.
At each point $\ula_R$ of $R$ there exist 
an open neighbourhood $\cU_R$ of $\ula_R$ and a unique pair (up to permutation)
$\te_1$, $\te_2$ of  regular sections of $B_R^\vee|_{\cU_R}$ such that

(a) For each point $\ulb$ of $\cU_R$ the kernels $\ker (\te_2 (\ulb))$ and $\ker (\te_2(\ulb))$
are both a line of $B_{R,\ulb} = B_{\sgm_R(\ulb)}$ and are orthogonal for the metric 
$\bg^{\sgm_R}$. 

(b) For each point $\ulb$ of $\cU_R$ the lines $\ker(\te_1(\ulb))$ and $\ker(\te_2 (\ulb))$ are the orthogonal "eigen-lines" 
of $\kp_R$, a local generator of $\cL_R := \cL^{\sgm_R,div}$ over $\cU_R$ (up to shrinking $\cU_R$).

(c) The local generator $\kp_R$ of $\cL_R$ writes over $\cU_R$ as  
\begin{equation}\label{eq:param-sum-squares}
\kp_R = \ve_1\cM_1 \te_1\tns\te_1 + \ve_2 \cM_2 \te_2\tns\te_2,
\end{equation}
with $\ve_1,\ve_2 \in \{-1,1\}$ and $\cM_1$, $\cM_2$ are  monomials in $V_\cL^\str \cup E_R$
locally generating $\sgm_R^*\cC_\cL$.

\medskip\noindent
2) Assume $\cL$ is constant along the fibers. Let $\Te$ be any invertible $\cO_S$-sub-module 
of $\Gm_S (B^\vee)$, not everywhere co-linear to $\cL$ and with co-support of codimension two 
if not empty. Let $\cC_\Te$ be its ideal of coefficients.

\smallskip
i) The mapping $\sgm_R$ (depending on $\Te$) factors as $\sgm_R = \sgm_1\circ\beta_R$, 
for a locally finite sequence of point blowings-up 
$\beta_R: (R,E_R) \to (S_1,E_1)$, so that 
Lemma \ref{lem:param-1} holds true.

ii) The ideal $\sgm_R^*\cC_\Te$ is principal and monomial in the 
snc-divisor $V_\Te := \cosupp(\sgm_R^*\cC_\Te)$ 
which is normal crossing with $E_R \cup V_\cL^\str \cup D_\cL^\str$.

iii) Let $\cL_R := \cL^{\sgm_R,div}$ and $B_R:=B^{\sgm_R}$. 
Let $\Te_1$ be the $\cO_R$-module of $\Gm_R (B_R^\vee)$ defined as $(\sgm_R^*\cC_\Te)^{-1} \Te^{\sgm_R}$ 
(with empty co-support). Let $\Te_1^{\perp,\bP}$ be the sub-module of $\Gm_R (\bP B_R^\vee)$ 
orthogonal to $\Te_1^\bP$ of $\Gm_R(\bP B_R^\vee)$ generated by $\Te_1$, 
for the fiber-metric $\bg^\vee\circ\sgm_R$ where $\bg^\vee$ is the fiber metric on $B^\vee$ induced by 
the metric $\bg$.  

Points (a), (b) and (c) of 1-ii) are satisfied by $\Te_1,\Te_1^{\perp,\bP}$ 
substituting 
$D_\cL^\str$ by $V_\Te$.
\end{proposition}
\begin{proof}
Suppose that $\cL$ is not constant along the fibers. We check that $\sgm_R := \sgm_1\circ\beta_R$ 
satisfies all the properties. We start after Lemma \ref{lem:param-2}.

\medskip
Let $\cL_R:=\cL_1^{\beta_R} =(\cL^{\sgm_R})^{div}$.

\medskip
If $\cL$ is not constant along the fibers, then the ideal $\beta_R^*I_{\cL_1}^{VD}$ 
is a principal $\cO_R$-ideal monomial in $VD_\cL$ while 
$\cQ_{\cL_1}^{\beta_R} = J_R\cdot\cQ_R$, where $J_R$ is a principal $\cO_R$-ideal and monomial 
in $VD_\cL$ and $\cQ_R$ is a $\cO_R$-invertible sheaf of regular sections of $S(2,B_R)$ whose 
zero locus is a regular two sheeted ramified covering of $R$. Indeed:
Let $\ula_R$ be a point of $R$ and let $\cU_R\times\R^2$ be a locally trivializing chart 
of $B_R$ with coordinates $(w,z,X,Y)$. We find that 
$$
\beta_R^{-1} (VD_{\cL_1}) :=\{\cM[b(X^2 - Y^2) + a XY] = 0\} , \mbox{ and } 
VD_{\cL_R} :=\{b(Y^2 - X^2) + a XY = 0\}
$$
where $\cM b$ or $\cM a$ is a local generator of the ideal $\beta_R^*I_{\cL_2}^{VD}$ over $\cU_R$.
Thus either $b$ or $a$ is a regular unit over $\cU_R$, so that (up to a rotation in the fibers if $b$ vanishes
at $\ula_R$) we find
$$
VD_{\cL_R} :=\{X^2 + 2cXY - Y^2= 0\}
$$
for a regular function $c$ over $\cU_R$. 
In other words $X^2 + 2cXY - Y^2$ is a local generator of $\cQ_R$ over $\cU_R$ and
$$
X^2 + 2cXY - Y^2 = \te_1(X,Y)\cdot \te_2(X,Y)
$$
for regular sections $\te_1,\te_2$ of $B_R^\vee|_{\cU_R}$, whose kernels are orthogonal
with respect to the metric $\bg_R = \bg\circ\sgm_R$ and 
are respectively directed by $[1:-c\pm\sqrt{1+c^2}] \in \bP B_R$. 
These kernels correspond to the "eigen-directions" of any local generator of $\cL_R$.

Since $\cL$ is not constant along the fibers, so is $\cL_R|_{\cU_R}$. Thus the regular directions fields 
$$
[1:-c(w,z)+\sqrt{1+ c^2(w,z)}] \; \mbox{ and } \;
[1:-c(w,z)-\sqrt{1+ c^2(w,z)}]
$$ 
diagonalize simultaneously over $\cU_R$ any generator of $\cL_R|_{U_R}$, so that over $\cU_R$ we get 
points $(a)$, $(b)$ and $(c)$ immediately. 

\medskip
Point 2). Since $\cL$ is constant along the fibers, so is $\cL_1$  implying that
the quadratic ideal sheaf $\cQ_{\cL_1}$ is the null section of $S(2,B_2)$.

The properties announced for $\Te_1$ and then $\Te_1^{\perp,\bP}$ are straightforward since 
the additional blowings-up $\beta_R : (R,E_R) \to (S_1,E_1)$ are just to "make nice" 
the module $\Te^{\beta_R,div}$.
\end{proof}

\begin{remark}\label{rmk:QR}
In the body of the proof of 1-ii), we show the existence of the invertible $\cO_R$-sheaf $\cQ_R$ 
of regular sections of $S(2,B_R)$, such that the restriction $\cQ_R|_{\cU_R}$ has a global section over $\cU_R$, namely 
$\te_1\cdot\te_2$. The sub-module $\cQ_R$ is a global object, unlike its local generators $\te_1$ and $\te_2$, any of which 
may not be glued to to yield a global object since there may be topological obstructions, such as the Euler class of $B$.
\end{remark}
A careful reading of the above proof, that is tracking down which 
ideals have been principalized and monomialized, also provides the additional following properties.
\begin{corollary}\label{cor:param-reg}
If $\cL$ is not constant along the fibers, we see that in point c), we further obtain

\smallskip
If $\cL_R$ is not degenerate then $\cM_1 = \cM_2$ is a monomial in $V_\cL^\str \cup E_R$
locally generating the ideal $\sgm_R^*\cC_\cL$.

If $\cL_R$ is generically non-degenerate, one of the 
function germ $\cM_1,\cM_2$ is a local monomial in $V_\cL^\str\cup E_R$
locally generating the ideal $\sgm_R^*\cC_\cL$, while the other one is a monomial in 
$V_\cL^\str \cup D_\cL^\str\cup E_R$ and is locally generating the $\cO_R$-ideal $\sgm_R^*I_\cL^D$.

If $\cL_R$ is degenerate, one of germs $\cM_1,\cM_2$ is a monomial in $V_\cL^\str \cup E_R$
locally generating the ideal $\sgm_R^*\cC_\cL$, and the other germ $\cM_1,\cM_2$ is identically zero. 
\end{corollary}
As a final consequence of Proposition \ref{prop:param-reg} (see \cite{KuPa,Gra2}) we find  
\begin{corollary}[see \cite{Gra2}]
Let $\cL$ be a non-zero invertible $\cO_S$-module of $\Gm_S (S(2,B))$. 
There exists a locally finite sequence of points blowings-up $\gm: (S',E') \to S$ 
such that for each point $\ula^\prime$ in $S'$, there exist a neighbourhood $U'$ of $\ula^\prime$ 
and two orthonormal and non-vanishing local sections $\xi_1,\xi_2:U'\to B^\gm$ such that 
at each point $\ulb^\prime$ of $U'$, the $2$-symmetric tensor $\kp'$ locally generating $\cL^\gm$ 
is a sum of squares in the basis $\xi_1(\ulb^\prime),\xi_2(\ulb^\prime)$.

Consequently when $\cL$ admits a global section $\kp$ over $S$, each "eigen-value" 
of the $2$-symmetric tensor $\kp^\gm$, that is the size of each generator $(\cM_i\theta_i)^2$
for $i=1,2$, is a monomial times a local unit nearby $\ula^\prime$, thus analytic.
\end{corollary}
%
%
%
%
%
%
%
%
%
%
%
%
%
%
%
%
%
%
%
%
%
%
%
%
%
%
%
%
%
%
%
%
%
%
%
%
%
%
%
%
\section{Resolution of singularities of plane singular foliations}\label{section:res-foliat}
We present here some elements of the classical reduction of singularities of plane foliations, 
aiming at our paper be as self-contained as can reasonably be. The 
material presented below is well known and can be found for instance in  \cite{Se,Dum,vdE,Ca1,CCD}

\medskip
Let $\cO_2:= \cO_{\R^2,\bo}$ be the local $\R$-algebra of regular function germs at 
$\bo$ the origin of $\R^2$, with maximal ideal $\bm_2$.
Let $\Omg_2^1$ be the $\cO_2$-module of regular differential $1$-form germs at $\bo$.

\smallskip
Let $\xi$ be a germ of regular vector field at the origin $\bo$ of $\R^2$. Given any regular local 
coordinates system $(x,y)$ centered at $\bo$, the vector field writes as $\xi = a (x,y) \dd_x + b(x,y) \dd_y$ 
where $a,b \in \cO_2$. Since we are only interested in foliations (phase portraits), 
up to dividing $\xi$ by $gcd (a,b)$, we can assume that $a$ and $b$ have no common factor
so that any vector field of the form $Unit\cdot \xi$ gives the same foliation as $\xi$. 
The vector field $\xi$ comes with (up to the multiplication by a regular unit) a unique dual regular 
differential form defining the same foliation, namely $\omg = b \rd x  - a \rd y$. 
Since $\gcd(a,b) =1$, we have 
\begin{equation}\label{eq:milnor-foliat}
\iota (a,b,\bo) := \dim_\R \cO_2/(a,b) < +\infty.
\end{equation}
\begin{definition}\label{def:foliation}
1) A \em germ of a plane foliation $\cF$ \em at the origin of $\R^2$ is the data of an invertible $\cO_2$-sub-module
$\cD_\cF$ of $\Omg_2^1$, which is finite co-dimensional at the origin, that is satisfying Equation 
(\ref{eq:milnor-foliat}). If $\cD_\cF$ is generated by $\omg = b\rd x - a\rd y$, the number $\iota (a,b,\bo)$ depends only
on the invertible sheaf generated by $\omg$, therefore we write $\iota(\cF,\bo)$ instead.
\\
2) Let $S$ be a regular surface. A foliation $\cF$ on $S$ is the data of a
non-zero $\cO_S$-invertible sub-module $\cD_\cF$ of $\Omg_S^1$ such that 
at each point $\bp$ of $S$ there exists a regular diffeomorphism germ $\phi: (S,\bp) \to (\R^2,\bo)$ 
such that the $\cO_2$-submodule $\phi_*\cD_\cF$ of $\Omg_2^1$ is generating a germ of plane foliation 
at $\bo$. 
\end{definition}
The invertible sub-module $\cD_\cF$ of $\Omg_2^1$ corresponding to the germ of foliation $\cF$
generated by the vector field germ $\xi$ is $\cO_2 \omg$, for $\omg$ its dual form. Point 1) of the 
Definition implies that the  germ of vector field $\xi$ (equivalently $\omg$) may only vanish at $\bo$, 
so that for $\bp$ close enough to $\bo$ but not $\bo$, we find $\iota(\cF,\bp) =0$.

\smallskip
Let $\cF$ be a foliation on a regular surface $S$. The \em ideal of coefficients of \em $\cF$ is the 
$\cO_S$-ideal $\cC_\cF$ obtained by the evaluation of any local generator of $\cD_\cF$ along the 
germs of regular vector field on $S$. The \em singular locus of $\cF$ is the sub-variety 
\em defined as 
\begin{center}
$\sing (\cF) := \cosupp \, \cC_\cF$, 
\end{center}
and is of codimension $2$, if not empty.
A plane foliation will be said \em singular \em if the sub-variety $\sing(\cF)$ is not empty.
\begin{definition}[\cite{Se,vdE,Dum,CCD,IlYa}]\label{def:simple-sing}
Let $\cF$ be a germ of plane foliation with an isolated singularity at the origin. 
The origin $\bo$ is called a \em simple singularity \em of 
$\cF$, if there exist local coordinates $(x,y)$ centered at the origin such that the local generator $\omega$ writes
\begin{center}
$\omega = \lambda y \rd x - \mu x \rd y + \theta$,
\vspace{4pt}
\end{center}
with $\lambda \in \R, \mu \in \R^*$, and $\mu^{-1}\lambda \notin \Q_{>0}$, 
and $\theta \in \bm_2^2 \Omega^1_2$, 
\end{definition}

\begin{definition}
Let $S$ be a regular surface and $\cF$ be a foliation 
on $S$. Let $\bp$ be a point of $S$. 
\\
- Let $(H,\bp)$ be a hypersurface germ, regular at $\bp$. It is called \em invariant (or non-di-critical) at $\bp$ 
with respect to $\cF$, \em if it is a union of leaves of $(\cF,\bp)$. If it is not invariant at $\bp$, it is called  
\em di-critical at $\bp$. \em
\\
- A normal crossing divisor $D$ of $S$ containing $\bp$ is \em invariant at $\bp$ w.r.t. $\cF$ \em if each 
irreducible component of $(D,\bp)$ is invariant at $\bp$ w.r.t. $\cF$.
\\
- If $\bp$ is a regular point of $\cF$, the foliation $\cF$ is \em normal crossing with $D$ at  $\bp$ 
\em if the germ $(D,\bp)$ is not invariant, and the union $\cL_\bp\cup D$, of the leaf $\cL_\bp$ through $\bp$ with $D$,
is the germ of a normal crossing divisor at $\bp$.  
\end{definition}
We introduce the following well known convenient
\begin{definition}\label{def:adapt-foliat}
A foliation $\cF$ of $S$ is \em adapted to the nc-divisor $D$ of $S$ \em if: 
\\
- Each singular point of $\cF$ is simple; 
\\
- At a regular point $\bp$ of $D$, regular point of $\cF$, the foliation $\cF$ is normal crossing with the germ $(D,\bp)$;
\\
- At a corner point $\bp$ of $D$, regular point of $\cF$, one local component of $(D,\bp)$ is invariant at $\bp$ w.r.t. $\cF$ and the 
other one is normal crossing with $\cF$ at $\bp$;
\\
- At a point $\bp$ of $D$, singular point of $\cF$, the germ $(D,\bp)$ is invariant at $\bp$ w.r.t. $\cF$.
\end{definition}
\noindent
{\bf Reminder on notations.}
\em 
In this section and in the following one, we recall that when $\te$ is a differential $1$-form on a manifold $N$
(or a sub-module of $\Omega_N^1$), the notation $\sgm^* \te$ means the pull back by a given regular mapping 
$\sgm:M\to N$ in the sense of differential topology, that is $\sgm^*\te :=\te\circ D\sgm$ where $D\sgm$ 
is the differential mapping of $\sgm$. 
\em 
 
\medskip
Let $\cF$ be a singular foliation at the origin $\bo$ of $\R^2$.

Let $\pi: S_0:=[\R^2,\bo] \to \R^2$ be the blowing-up of the origin $\bo$ and let $E_0$ 
be the exceptional curve $\pi^{-1}(\bo)$. Let $I_{E_0}$ be the $\cO_{S_0}$-ideal 
of the regular function germs vanishing on $E_0$. There exists a largest  positive integer $m$
such that $ I_{E_0}^{-m} \cdot \pi^*\cD_\cF$ is a $\cO_{S_0}$-invertible sub-module of $\Omg_{S_0}^1$,
which is finite co-dimensional everywhere.
\begin{remark}\label{rem:noether}
Let $\beta: (S,E) \to (\R^2,\bo)$ be a finite composition of point blowings-up, where $S$ is a regular surface and 
$E := \beta^{-1}(\bo)$ is a snc divisor.
The ideal $\cC_{\beta^*\cD_\cF}$ of coefficients of 
the $\cO_S$-sub-module $\beta^*\cD_\cF$ decomposes into a product of $\cO_S$-ideals $J\cdot K$, where 
the ideal $J$ is principal and monomial in the exceptional divisor $E$, while the ideal $K$ (with co-support in $E$) 
has finite co-dimension, namely $\dim_\R \cO_{S,\bp}/K < + \infty$, for any point $\bp$ of $S$. 
\end{remark}
Remark \ref{rem:noether} leads to the following
\begin{definition}\label{def:pull-back-foliat}
Let $\beta: (S,E) \to (\R^2,\bo)$ be a finite composition of point blowings-up.
\\
Using the notations of Remark \ref{rem:noether}, the \em pulled-back foliation $\beta^*\cF$ of $\cF$, \em is given
by the invertible $\cO_S$-sub-module $\cD_{\beta^*\cF}:= J^{-1}\beta^*\cD_\cF$ of $\Omega_S^1$.
For a local generator $\omg$ of $\cF$ at $\bo$, every point $\bp$ of $S$ admits an open neighbourhood $U$
over which there exist a monomial $\cM$ in the exceptional divisor $E$ (generating $J$ locally over $U$) and 
a local generator $\te$ of $\cD_{\beta^*\cF}$ over $U$ such that $\cM\te = \beta^*\omg$. 
The local generator $\cM^{-1}\beta^*\omg$ is called \em the strict transform of $\omg$ 
under the blowing-up $\beta$. \em
\end{definition}
We can now present the theorem of reduction of singularities of singular plane foliation (\cite{Se,vdE,Ca1,CCD})
in the form which is most convenient for our later use. 
The global reduction of singularities of plane foliation being deduced from the local one,
we only present the latter one in the following classical form:
\begin{theorem}[\cite{Se,vdE,Dum,Ca1,CCD,IlYa}]\label{thm:desing-foliat}
Let $\cF$ be a germ of singular plane foliation at the origin $\bo$ of $\R^2$.
There exists a finite composition of points blowings-up $\pi: (S',E') \to (\R^2,\bo)$ such that 
each point of $\sing(\cF',E')$ of the pulled-back foliation $\cF':=\pi^* \cF$ is a simple singularity of $\cF'$
adapted to $E'$. 

\smallskip
Moreover if $\beta$ is the point blowing-up $(S'',E'') \to (S',E'\cup\{\bp'\})$ with center $\bp'$, 
then $\cF'' := \beta^*\cF'$ only admits simple singularities adapted to $E''$.   
\end{theorem}

We end the section with the normal form of a local generator of a "reduced" germ of singular 
plane foliation $\cF$ with simple singularities adapted to an exceptional divisor $E$ as in Theorem \ref{thm:desing-foliat}. 
By Remark \ref{rem:noether}, the exceptional divisor $E$ contains  the singular locus $\sing (\cF)$.

\smallskip\noindent
$\bullet$ Let $\bp \notin E$, then $\omega (\bp) \neq 0$. 

\smallskip\noindent
$\bullet$
Suppose $\bp \in E \setminus\sing (\cF)$. There exist local coordinates $(u,v)$ centered at $\bp$ such that 
\begin{center}
$\{u=0\} \subset (E,\bp) \subset  \{uv =0\}$ \; and \; $\omega (\bp) \neq 0$. 
\end{center}
If $(E,\bp) =\{u=0\}$ is invariant for $\cF$, a local generator is of the form $\omg =\rd u + u(\cdots)\rd v$.
\\
If $(E,\bp) =\{u=0\}$ is normal crossing with $\cF$, a local generator is of the form $\omg = \rd v$ 
(up to a change of coordinate of the form $\brv = v + F$, with $F\in \bm_\bp$).  
\\
If $(E,\bp) =\{uv=0\}$ and if a local generator is of the form $\rd u + Unit\cdot \rd v$, we check 
that blowing-up $\bp$ will give a local generator of the pulled-backed foliation such that at each 
of the new two corners one of the new exceptional divisor and the strict transform of the corresponding old 
component through $\bp$, is invariant and the other one is di-critical. Thus (up to blowing-up the point $\bp$), 
we deduce that, up to permuting $u$ and $v$, a local generator is given by $\rd u + u (\cdots)\rd v$ 
(see Lemma \ref{lem:adapt-1} and Lemma \ref{lem:adapt-2} for details).

\smallskip\noindent
$\bullet$
Suppose $\bp \in E \cap\sing (\cF)$. Thus each component of $E$ must be invariant.
There exist local coordinates $(u,v)$ centered at $\bp$ such that 
$\{u=0\} \subset (E,\bp) \subset  \{uv =0\}$ and there exists another set of local coordinates $(x,y)$ centered at $\bp$ 
such that $\omg = \lambda x\rd y - y \rd x + \te$ where $\te \in \bm_\bp^2
\Omega_\bp^1$ and $\lambda \notin \Q_{>0}$.
\\
If $(E,\bp) =\{uv=0\}$, since it is non-di-critical, we deduce also that, up to permuting $u$ and $v$, we get
$\omg = v\rd u + u(\cdots)\rd v$.
\\
If $(E,\bp) =\{u=0\}$, then a local generator writes as $\omega = uA\rd v + (uB + v^k \phi(v))\rd u$, 
with $A,B \in \cO_\bp$, where $\phi$ is an analytic unit in a single variable and $k=1$ if $A(\bp) =0$.
When $A(\bp) \neq 0$, a local generator is of the form  
\begin{equation}\label{eq:form-1}
u\rd v + (v^k \phi(v)+uB)\rd u
\end{equation}
and when $A(\bp) = 0$, it is of the form 
\begin{equation}\label{eq:form-2}
(v+b_0 u)\rd u + u\te
\end{equation}
where $\te\in\bm_\bp\Omg_\bp^1$ and $B\in\cO_\bp$ and $b_0\in \R$.
%
%
%
%
%
%
%
%
%
%
%
%
%
%
%
%
%
%
%
%
%
%
%
%
%
%
%
%
%
%
%
%
%
%
%
%
%
%
\section{Pairs of singular foliations and singular foliation adapted to nc-divisors}\label{section:pairs-foliat}
The material presented in this section is known folklore of desingularization of singular plane 
foliations (see \cite{Ca1,CCD,IlYa}). In order to have our paper reasonably self-contained, 
we present the results, that we will use in the demonstration of Theorem \ref{thm:main}, with proofs.

\medskip 
Any singular foliation can be resolved 
into a singular foliation with simple singularities adapted to the exceptional divisor of the 
resolution (Section \ref{section:res-foliat}). Simple adapted singularities 
transform under blowings-up either in simple adapted singularities or in regular points. 

\smallskip
An additional property of the resolution of a singular foliation to wish for 
is a good behavior of the foliation with respect to (the pull-back of) a given curve. 
A consequence of Definition \ref{def:adapt-foliat} is the next
\begin{lemma}[see also {\cite[Section 8]{IlYa}}]\label{lem:adapt-1}
Let $D$ be a nc-divisor of a regular surface $S$ and let $\cF$ be a foliation on $S$ adapted to $D$.
Let $\beta: (S',E') \to S$ be the blowing-up with center a given point $\bp$ of $S$.
The pulled-back foliation $\beta^*\cF$ is adapted to the nc-divisor $\beta^{-1} (D)$.
\end{lemma}
\begin{proof}
If $\bp$ does not belong to $D$, it is true. 

\medskip
Let $(u,v)$ be local coordinates centered at the corner point $\bp$ of $D$ and adapted to $D$, 
thus $(D,\bp) = \{uv=0\}$. We use the normal forms presented at the end of
Section \ref{section:res-foliat}.

\smallskip
If $\bp$ is a singular point of $\cF$ adapted to $D$, then, up to permuting $u$ and $v$, 
a local generator of $\cF$ at $\bp$ is $\te = v\rd u  + u A \rd v$ with $A\in\cO_\bp$ such 
that $- A(\bp)$ is not in $\Q_{>0}$.

In the chart $(x,y) \to (x,xy)$ of the blowing-up $\beta$ of the point $\bp$, we find 
$$
\beta^{-1} (D) = \{xy =0\}, \; \mbox{ and } \;\beta^*\te = x \cdot Unit \cdot (y \rd x + x B \rd y),
$$
so that in this chart $\beta^*\cF$ is adapted to $\beta^{-1} (D)$. 
In the chart $(x,y) \to (xy,y)$, we check similarly that
$\beta^*\cF$ is also adapted to $\beta^{-1} (D)$. 

\smallskip
If $\bp$ is a regular point of $\cF$ then, up to permuting $u$ and $v$, 
a local generator of $\cF$ at $\bp$ is $\te = \rd u  + u A \rd v$ with $A\in\cO_\bp$ such 
that $- A(\bp)\notin \Q_{>0}$.

In the chart $(x,y) \to (x,xy)$ of the blowing-up $\beta$ of the point $\bp$, we find 
$$
\beta^{-1} (D) = \{xy =0\}, \; \mbox{ and }  \;   \beta^*\te = Unit \cdot (\rd x + x B \rd y)
$$
so that in this chart $\beta^*\cF$ is adapted to $\beta^{-1} (D)$. 
In the chart $(x,y) \to (xy,y)$, we get
$$
\beta^{-1} (D) = \{xy =0\}, \;  \mbox{ and }  \;   \beta^*\te = Unit \cdot (y \rd x + x B \rd y)
\; \mbox{ with } \; B(0,0) = 1
$$
so that in this chart $\beta^*\cF$ is also adapted to $\beta^{-1} (D)$. 

\medskip
When $\bp$ is regular point of $D$, similar computations will lead to the stated conclusion, using 
the normal forms at the end of Section \ref{section:res-foliat}.
\end{proof}
Let $D$ be a (s)nc-divisor of $S$ and let $\cF$ be a foliation on $S$.   
Let $NA(\cF,D)$ be the subset of points of $S$ where the foliation $\cF$ is not adapted to the germ of 
$D$ at this point. It is a real analytic set which is isolated, when not empty.
\begin{lemma}[see also {\cite[Section 8]{IlYa}}]\label{lem:adapt-2}
There exists $\sgm_1: (S_1,E_1) \to S$, a locally finite sequence of points blowings-up, such 
that the pulled-back foliation $\sgm_1^*\cF$ is adapted to the (s)nc-divisor $E_1\cup D^\str$. 
\end{lemma}
\begin{proof}
We resolve the singularities of $\cF$ with $\beta: (S',E') \to S$ a locally finite sequence of 
blowings-up, so that $\cF'$ the pulled-back foliation $\beta^*\cF$ is 
adapted to $E'$. The strict transform $D'$ of $D$ is a (s)nc-divisor which is normal 
crossing with $E'$ (up to possibly finitely many further points blowings-up). The 
intersection  $D'\cap E'$ consists only of isolated points. 

\medskip
Since the locus $NA(\cF',D')$ of non-adapted point is isolated, we can suppose that it is 
reduced to the single point $\{\bp'\}$. Let $(u',v')$ be local coordinates centered at $\bp'$
and adapted to $D'$ so that $\{v'=0\} \subset (D',\bp') \subset \{u'v'=0\}$. 

\smallskip 
1) Let $\bp'$ be a point of $D'\cap E'$, so that it is a regular point of $D'$. 
Suppose that $u'$ is such that $(E',\bp'):= \{u'=0\}$.

\smallskip
a) Suppose that $\bp'$ is a regular point of $\cF'$. 
Thus the leaf through $\bp'$ is tangent to $D'$ (thus normal crossing with $E'$) 
while all the nearby ones are normal crossing with $D'$.
A local generator of $\cF'$ is of the form 
\begin{center}
$Unit\cdot[(u')^l +v'(\cdots)] \rd u' + \rd v'$
\end{center}
for $l$ a positive integer. 

The exceptional curve $C''$ obtained by blowing-up the point $\bp$
is a maximal invariant curve 
of the strict transform $\cF''$ of $\cF'$ and is normal crossing with $D''$, the strict transform
of $D^\prime$.
In the chart $(u'',v'') \to (u'',u''v'')$, we find $D'' = \{v'' = 0\}$ with $C'' = \{u'' = 0\}$, 
so that a local generator at $\bp''= (0,0)$ of $\cF''$ is of the form
\begin{center}
$Unit\cdot[(u'')^{l-1} +v''(\cdots)] \rd u'' + \rd v''$.
\end{center}
The strict transform $D''$ does not meet with the domain of the other 
blowing-up chart. We are in the same situation with $\bp'',C''\cup E',D''$ as we were with $\bp',E',D'$, 
but with the exception that the exponent $l$ in the local generator of $\cF''$ at $\bp''$ has dropped by $1$.
With further $l-1$ point blowings-up (each center being the intersection of the latest strict transform of $D'$
with the latest created exceptional curve) we see in the (only interesting) chart $(\bu,\bv) \to (\bu,\bu^{l-1}\bv) 
= (u'',v'')$), that a local generator of the pull-back of $\cF''$ at the terminal point 
of this process $\bp^{l} = (0,0)$ is of the 
form $\rd \bu + \omega$, where $\omega \in \bm_{\bp^l}\cO_{\bp^l}$.
Thus there exists $\beta_1: (S_1',E_1') \to (S',E')$ factorizing through the blowing-up of the point $\bp'$
so that $E_1' = E' \cup \beta_1^{-1} (\bp')$ so that there exists a neighbourhood $\cU_1'$ of the 
exceptional divisor $\beta_1^{-1} (\bp')$ such that the restricted pulled-back foliation $(\beta_1^*\cF')|_{\cU_1'}$ 
is adapted to $(D^\str\cup E_1')\cap \cU_1'$ (note that our notational conventions 
imply that $D^\str = (D')^\str = (D'')^\str $ as well). 

\smallskip
b) Suppose now that $\bp'$ is a singular point of $\cF'$.
The curve $(E',\bp')$ is invariant while, any leaf of $\cF'$ through any 
point $\bq\in D'$ nearby $\bp'$ is normal crossing with $D'$. 
Under the current hypotheses, a local generator $\te$ of $\cF'$ is of the form 
\begin{center}
$[Unit \cdot (v')^k + (u')^l + u'v' (\cdots)]\rd u' + u'A'\rd v'$ with $k,l\in \N_{\geq 1}$ and $A'\in \cO_{\bp'}$.
\end{center}
We blow-up the point $\bp'$. In the chart $(u'',v'') \to (u''v'',v'')$, the strict transform of $D'$ is empty.
In the chart $(u'',v'') \to (u'',u'' v'')$, the strict transform of $D'$ is $D'':=\{v'' =0\}$ and 
the exceptional divisor is just $C'':=\{u'' =0\}$. Let $\bp''= (0,0)$. If $\cF''$ is the pull-back
of $\cF'$, then a local generator of $\cF''$ is of the form
\begin{center}
$[Unit\cdot(v'')^k + (u'')^{l-1} + u''v''(\cdots)]\rd u'' + u'' A'' \rd v''$ where $A'' \in \cO_{\bp''}$. 
\end{center}
we conclude as in case 1) with further $l-1$ points blowings-up. 

\smallskip 
2) Suppose that $\bp'$ lies in $D' \setminus E'$. Thus it is a regular point of $\cF'$. 

A simple computation, distinguishing between a regular point or a corner point of $D'$, 
shows that after blowing-up the point $\bp'$ we are in the situation 1).
\end{proof}

\medskip
We can now turn to pairs of foliations.
Suppose given two (possibly singular) foliations $\cF_1$ and $\cF_2$ on the regular 
surface $S$. 
Let $\cD_i$ be the invertible sub-module of $\Omega_S^1$ corresponding to $\cF_i$, for 
$i=1,2$. Suppose that these two foliations are transverse to each other on an analytic Zariski 
dense open set of $S$. Let $\cD_{1,2}$ be the invertible $\cO_S$-module of $\Omega_S^2$ generated by 
$\cD_1\wedge \cD_2$. Let $\cC_{1,2}$ be the ideal of its coefficients, so that 
$\cD_{1,2} = \cC_{1,2}\Omega_S^2$.  Let $\Sgm(\cF_1,\cF_2)$ be its co-support, 
the tangency locus of $\cF_1$ and $\cF_2$. We get the following expected
\begin{proposition}\label{prop:pairs-good-position}
Let $C$ be a connected real analytic curve in $S$ which is neither contained in $\Sgm(\cF_1,\cF_2)$
nor contains any component of $\Sgm(\cF_1,\cF_2)$. 
There exists $\tau: (S',E') \to (S,E)$ a locally finite sequence of point blowings-up such that 

\smallskip
i) The strict transform $C^\str$ of $C$ is a snc-divisor which is normal crossing with $E'$.

ii) The ideal $\tau^*\cC_{1,2}$ is principal and monomial in its co-support, namely the snc-divisor 
$\tau^{-1}(\Sgm (\cF_1,\cF_2))$, and is normal crossing with $E'\cup\cC^\str$.

\smallskip
iii) Let $\cF_i'$ be the pull-backed foliation of $\cF_i$, for $i=1,2$. 
The singularities of $\cF_i'$ are adapted to the snc-divisor $D:=C^\str \cup E'\cup \tau^{-1}(\Sgm (\cF_1,\cF_2))$, 
and each component of $D$ is either invariant or di-critical.
We can further demand that each foliation $\cF_1',\cF_2'$ is adapted to the snc-divisor 
$D\cup E'$.

iv) Let $\cD_i'$ be the $\cO_{S'}$-sub-module of $\Omega_{S'}^1$ associated with $\cF_i'$. Thus 
$\cD_1'\wedge \cD_2' = \cJ_{1,2}' \Omega_{S'}^2$ with $\cJ_{1,2}'$ a principal ideal and monomial  
in the snc-divisor $\cosupp(\cJ_{1,2}')$ (contained in the snc-divisor $\subset (\tau^{-1}(\Sgm (\cF_1,\cF_2)))^\str$), 
and contains the ideal $\tau^*\cC_{1,2}$.
\end{proposition} 
\begin{proof}
First we resolve the singularities of the curve $C$ by means of a locally finite sequence of points blowings-up, 
namely $\sgm_1:(S_1,E_1) \to S$. 

\medskip
Second we principalize and monomialize the ideal $\sgm_1^*\cC_{1,2}$ by means of a locally finite sequence 
of point blowings-up $\beta_2: (S_2,E_2) \to (S_1,E_1)$ so that  $\cosupp (\sgm_2^*\cC_{1,2})$ is a snc-divisor 
which is normal crossing with $E_2\cup C^\str$ and where $\sgm_2:= \sgm_1\circ\beta_2$. 

\medskip 
Third, by means of $\beta_3: (S_3,E_3 = E_2 \cup E_3') \to (S_2,E_2)$, a locally finite sequence of point
blowings-up  where $E_3'$ is the new exceptional divisor (and keeping denoting $E_2$ for the strict transform 
$E_2^\str)$, the pulled-back foliation $\cF_i^3:= \beta_3^*(\sgm_2^*\cF_i)$ only have singularities adapted 
to $E_3'$ where $i=1,2$ (and is regular at each point of $E_2\setminus E_3'$).

\medskip
We do further point blowings-up $\beta_4: (S_4,E_4) \to (S_3,E_3)$ so that the foliations $\beta_4^*\cF_i^3$ 
are adapted to the snc-divisor $E_4 \cup C^\str\cup \Sgm (\cF_1,\cF_2)^\str$, which is possible thanks to Lemma \ref{lem:adapt-1}
and Lemma \ref{lem:adapt-2}.
Then we define $\cF_i':=\beta_4^*\cF_i^3$ and $\tau:=\sgm_2\circ\beta_3\circ\beta_4$.

\medskip
Last to get point iv) we have

$$\tau^*\cD_1 = \cI_1 \cD_1'\, , \; \;\tau^*\cD_2 = \cI_2 \cD_2 \, , \;\;
\tau^*(\cD_1 \wedge\cD_2) = \tau^*(\cC_{1,2} \Omega_S^2) = \tau^*\cC_{1,2} \Omega_{S_4}^2
$$
where $\cI_1$ and $\cI_2$ are principal and monomial ideals in some snc divisors contained in
$E_4 \cup C^\str\cup \Sgm^\str$. From point ii) we deduce point iv) since
$\tau^*\cC_{1,2} = \cI_1\cdot\cI_2\cdot\cI_{1,2}'$.
\end{proof}
The next proposition is of interest for our main result,
but requires a few preparations.  
\begin{definition}[see \cite{Sa}]
Let $E$ be a normal crossing divisor of $S$. 
Let $\ula \in S$ and let $\{h=0\}$ be a local reduced equation of $(E,\ula)$ . 
A meromorphic differential $q$-form $\omega$ is \em logarithmic (along $(E,\ula)$) \em if $h\omega$ and
$h \rd \omega$ are both regular $q$-forms.
\end{definition}  
We denote $\Omg_S^q (\log E)$ the $\cO_S$-module of $q$-logarithmic forms along $E$.

If $\bp$ does not lie in $E$, there exists a neighbourhood $\cU$ 
of $\bp$ such that 
\begin{center}
$\Omg_\cU^1 = \Omg_S^1|_\cU = \Omg_S^1 (\log E)|_\cU$.
\end{center}
If $\bp$ is a regular point of $E$, we can find local coordinates $(u,v)$ at $\bp$ adapted to $E$
for which $(E,\bp) = \{u=0\}$, there exists a neighbourhood $\cU$ 
of $\bp$ such that 
\begin{center}
$\Omg_S^1 (\log E)|_\cU = \cO_\cU \rd v + \cO_\cU \rdlg u$,  where 
$\displaystyle{\rdlg u := \frac{\rd u}{u}}$.
\end{center} 

If $\bp$ is a corner point of $E$, there are local coordinates $(u,v)$ at $\bp$ adapted to $E$
such that $(E,\bp) = \{uv=0\}$, there exists a neighbourhood $\cU$ 
of $\bp$ such that 
\begin{center}
$\Omg_S^1 (\log E)|_\cU = \cO_\cU \rdlg u + \cO_\cU \rdlg v$. 
\end{center}

\smallskip
Since $\Omg_S^1$ is a sub-module of $\Omg_S^1 (\log E)$, any sub-module $\Te$
of $\Omg_S^1$ is a sub-module of  $\Omg_S^1(\log E)$. A local \em logarithmic 
generator of $\Te$ \em is a local generator of $\Te$ as a sub-module of $\Omg_S^1 (\log E)$. 
The \em ideal of logarithmic coefficients of $\Te$ \em is the ideal $\cC_\Te^{\log}$ 
locally generated by the logarithmic generator of $\Te$  
evaluated along local regular vector fields on $S$. Note that if $\cC_\Te$ is the ideal of coefficients of $\Te$, then
$\cC_\Te^{\log} \subset \cC_\Te$.
 
\medskip
Assume Proposition \ref{prop:pairs-good-position} is satisfied.

\smallskip
Let $\bp'$ be a corner point of $E'$. Let $(u,v)$ be local coordinates centered at $\bp'$ 
and adapted to $E'$. For each $i$, the foliation $\cF_i'$ has a local generator 
at $\bp'$ either of the form $\rd u + u (\cdots) \rd v$ (up to permuting $u$ and $v$) or 
$\bp'$ is an adapted singularity of $\cF_i'$. 
Let  $\te_i$ be a local generator of $\cD_i'$. The \em logarithmic generator 
$\te_i^{\log}$ of $\cD_i'$ associated to $\te_i$ \em is defined as follows:
\begin{itemize}
\item if $\te_i = \rd u + u(\cdots)\rd v$ then $\te_i^{\log}: = u^{-1} \te_i = \rdlg u + v (\cdots) \rdlg v$,
\item if $\te_i = v \rd u + u (\cdots)\rd v$ then $\te_i^{\log}: = u^{-1}v^{-1} \te_i = \rdlg u +  (\cdots) \rdlg v$.
\end{itemize}
Note that, in each case, the logarithmic $1$-form $\te_i^{\log}$ is nowhere vanishing.

Let $\cM_i$ be a local monomial (in $E'$) generating the ideal $\tau^*\cC_{\cD_i}$ and let $\cM_i^{\log}$ be a local generator
of $\cC_{\tau^*\cD_i}^{\log}$, the logarithmic coefficient ideal of the total transform $\tau^*\cD_i$. 
According to the two cases to distinguish, we either find that $\cM_i^{\log} = u\cdot\cM_i$ or, respectively,
that $\cM_i^{\log} = uv\cdot\cM_i$.
\begin{proposition}\label{prop:pairs-order}
Continuing Proposition \ref{prop:pairs-good-position}, let $\bp'$ be a corner point 
of the exceptional divisor $E'$.

v) If $\bp'$ is an adapted singular point of $\cF_1'$ and $\cF_2'$ such that the ideals 
$\cC_{\tau^*\cD_1}$ and $\cC_{\tau^*\cD_2}$ are not ordered at $\bp'$, there exists 
a locally finite sequence of point blowings-up $\beta'':(S'',E'' = E'\cup E_{\beta''}) \to (S',E')$ 
such that at each corner point of $E_{\beta''}$ the ideals $\cC_{(\tau\circ\beta'')^*\cD_1}$ and 
$\cC_{(\tau\circ\beta'')^*\cD_2}$  are ordered. Consequently so are the ideals of logarithmic coefficients 
$\cC_{(\tau\circ\beta'')^*\cD_1}^{\log}$ and $\cC_{(\tau\circ\beta'')^*\cD_2}^{\log}$.

vi) If $\bp'$ is a regular point of $\cF_1'$ such that the ideals 
$\cC_{\tau^*\cD_1}$ and $\cC_{\tau^*\cD}$ are not ordered at $\bp'$, there exists 
$\beta'':(S'',E'' = E'\cup E_{\beta''}) \to (S',E')$, a locally finite sequence of point blowings-up, 
such that at each corner point of $E_{\beta''}$ the ideals $\cC_{(\tau\circ\beta'')^*\cD_1}$ and 
$\cC_{(\tau\circ\beta'')^*\cD_2}$  are ordered. Thus the ideals $\cC_{(\tau\circ\beta'')^*\cD_1}^{\log}$ 
and $\cC_{(\tau\circ\beta'')^*\cD_2}^{\log}$ are also ordered.
\end{proposition}
\begin{proof}
We recall  that $\cJ_i:= \cC_{\tau^*\cD_i}$, so that 
$\cC_{\tau^*\cD_1}^{\log} = (uv)\cdot\cJ_i$.

\smallskip\noindent
Suppose $\bp'$ is an adapted singularity of 
both foliations. 
Let $(u,v)$ be local coordinates adapted to $E'$ at $\bp'$, so that
$(E',\bp') = \{uv = 0\}$.  Thus $\cJ_i$ is locally generated by $u^{p_i}v^{q_i}$ for non-negative 
integers $p_i,q_i$ and $i=1,2$. Suppose that $p_1<p_2$ and $q_2<q_1$.

Let $\te \in \Omg_{\bp'}^1$ such that $\bp'$ is a singularity adapted to $(E',\bp')$.
Thus we can assume (up to permuting $u$ and $v$) that $\te = v\rd u  - \lambda u\rd v + uv \eta$ 
where $\eta\in \Omg_{\bp'}^1$ and $\lambda \notin \Q_{>0}$. Let $\gm$ be the blowing-up of $\bp'$.
At any (of the two) corner point $\bp''$ of $E'':=E'\cup \gm^{-1}(\bp')$, let $z$ be a 
reduced equation of $\gm^{-1}(\bp')$. We find that $\sgm^* \te = z \cdot \te'$ for $\te'\in \Omg_{\bp''}^1$
and (following Lemma \ref{lem:adapt-1}) such that $\bp''$ is an adapted singularity of $\te'$. 

Let $\omg_i$ be a local generator of $\cF_i$. Let $\bp''$ be a corner point of $E''$.
Let $(t,z)$ be local coordinates at $\bp''$ adapted to $E''$ such that 
$(\gm^{-1}(\bp'),\bp'') = \{z=0\}$.
Thus we deduce that $(\tau\circ\gm)^*\omg_i = t^{r_i} z^{r_i +s_i +1} \omg_i'$ where $(r_i,s_i) = (p_i,q_i)$ 
or $(q_i,p_i)$ and for a local generator $\omg_i'$ of $\gm^*\cF_i'$ which has an adapted singularity at $\bp''$. 
Thus either $(r_1 + s_1 + 1) - (r_2 + s_2 +1)$ and $r_1 - r_2$ have the same sign or 
$|(r_1 + s_1 + 1) - (r_2 + s_2 +1)| < |r_1 - r_2|$.   

With further finitely many point blowings-up, at any corner point $\bq$ of the exceptional divisor
lying over $\bp'$, the coefficients ideals of the total transform of $\cD_1$ and $\cD_2$ are ordered.

\medskip\noindent
Let $\bp'$ be a corner point of $E'$ such that it is an adapted singularity of 
only one of the two foliations, say $\cF_2$. Let $(u,v)$ be local coordinates adapted to $E'$, 
thus $(E',\bp') = \{uv = 0\}$ and such that a local generator $\te_1$ of $\cF_1'$ is of the form  
$\rd u + uC\rd v$, so that $G:= \{u=0\}$ is invariant for $\cF_1'$ and $H:=\{v=0\}$ is di-critical. 
Let just write a local generator $\te_2$ of $\cF_2'$ as $vA \rd u + u B \rd v$ and (at least) one 
of the function germ $A,B$ is a local analytic unit.  

Suppose that for $i=1,2$ the ideal $\cJ_i$ is locally generated by $u^{p_i}v^{q_i}$ for non-negative 
integers $p_i,q_i$ and $i=1,2$ with $(p_1 -p_2)(q_1-q_2) <0$.

Let $\gm$ be the blowing-up with center $\bp'$. Note that for $i=1,2$ the logarithmic generator 
$\te_i^{\log}$ is of the form $a_i \rdlg u + b_i \rdlg v$, and at least one of the function germs 
$a_i,b_i$ is a local analytic unit. 

We observe that for $i=1,2$ that the pull-back $\gm^*\te_i^{\log}$  of $\te_i^{\log}$ is of the form 
$a_i' \rdlg u + b_i' \rdlg v$ and one of the function germs $a_i',b_i'$ is a local analytic 
unit. Thus the ideal of logarithmic coefficients of $(\tau\circ\gm)^*\cD_i$ is the pull-back 
$\gm^*(\cC_{\tau^*\cD_i}^{\log})$ of the ideal of logarithmic coefficients of $\tau^*\cD_i$.
We have thus replaced a problem of foliations into a problem on (principal and monomial) ideals to order,
which can be achieved by finitely many point blowings-up.  

\smallskip\noindent
If $\bp'$ is a corner point of $E'$ such that it is a regular point of both foliations $\cF_1'$ 
and $\cF_2'$, then we reach the conclusion as in the case of a single regular foliation.
\end{proof}
We end-up the section with some pairs of normal forms at points of $\Sgm (\cF_1',\cF_2',)$.

\smallskip\noindent
$\bullet$ Let $\bp$ be a regular point of $\Sgm (\cF_1',\cF_2')$ and let $H$ be the component containing $\bp$.
Let's pick local coordinates $(u,v)$ at $\bp$ such that $(H,\bp) = \{u=0\}$. Let $\te_i$ be a local 
generator of $\cF_i'$, for $i=1,2$. We find that $\te_1 \wedge \te_2 = Unit\cdot u^m \rd u\wedge \rd v$,
for a positive integer $m$. Note that $H$ is either invariant for $\cF_1'$ and $\cF_2'$ or is di-critical 
for $\cF_1'$ and $\cF_2'$.

\smallskip\noindent
{\bf Case 1.} \em Suppose $\te_1 (\bp) \neq 0$ and $\te_2 (\bp) \neq 0$. \em 
Then we check that $\te_2 = Unit\cdot\te_1 + u^m\omega$ with $\omg$ such that $\omg = \rd v$
if $H$ is invariant for $\te_1$and $\omg = \rd u$ if $H$ is di-critical for $\te_1$.

\smallskip\noindent
{\bf Case 2.} \em Suppose $\te_1 (\bp) \neq 0$ and $\te_2 (\bp)= 0$. \em 
Thus $H$ is invariant for $\cF_2'$. We get that $\te_1 = \rd u + uB_1 \rd v$ while 
$\te_2 = (v^k\phi(v) + uA_2) \rd u + u B_2 \rd v$, for function germs $A_2,B_1,B_2$ such that 
$uB_2 - uB_1(v^k\phi(v) + u A_2) = u^m$, so that $\te_2 = (v^k\phi (v) + uA_2) \te_1 + u^m \rd v$.
Note that $m=1$ if and only if $B_2$ is a unit.

\smallskip\noindent
{\bf Case 3.} \em Suppose $\te_1 = u\rd v + (v^{k_1}\phi_1(v) +uB_1)\rd u$ and $\te_2 (\bp)= 0$. \em 
Thus $H$ is invariant for both foliations and necessarily $m\geq 2$. We also write $\te_2 = u(v^{k_2}\phi_2(v) + uB_2) \rd u + u A_2 
\rd v$ so that $(v^{k_2}\phi_2 + uB_2) - A_2(v^{k_1}\phi_1 + uB_1) = u^{m-1}$, thus 
$\te_2 = u^{m-1} \rd v +  A_2\te_1$.

\smallskip 
There would be another case to consider, but the situation in which we will use these normal forms and their 
behavior is, as we will see in Appendix \ref{section:normal-form}, not generic. We will deal with this last situation in due time.

\medskip\noindent
$\bullet$ 
Let $\bp$ be a corner point of $\Sgm (\cF_1',\cF_2',)$.
Let $(u,v)$ be local coordinates centered at $\bp$ such that $(\Sgm(\cF_1',\cF_2'),\bp) = \{uv=0\}$. 
Let $\te_i$ be local generator of $\cF_i'$, for $i=1,2$. Thus we find that $\te_1 \wedge \te_2 = Unit \cdot u^m v^n 
\rd u\wedge \rd v$, for positive integers $m$ and $n$. 

\smallskip\noindent
{\bf Case 4.} \em Suppose $\te_1 (\bp) \neq 0$ and $\te_2 (\bp) \neq 0$. \em 
Up to permuting $u$ and $v$, we can find $\te_1$and $\te_2$ such 
that $\te_1 = \rd u + u (\cdots) \rd v$ and $\te_2 = Unit \cdot \te_1 +  u^mv^n \rd v$.

\smallskip\noindent
{\bf Case 5.} \em Suppose $\te_1 (\bp) \neq 0$ and $\te_2 (\bp) = 0$. \em 
We find $\te_1$ such that $\te_1 = \rd u + u (\cdots) \rd v$, up to permuting $u$ and $v$. 
The point $\bp$ is a singularity of $\cF_2'$ adapted to $\Sgm(\cF_1',\cF_2')$ and 
$\te_2 = w\rd z + z (\cdots) \rd w$ for $(w,z) = (u,v)$ or $(w,z) = (v,u)$. 
We deduce that $\te_2 = Unit\cdot v\te_1 + u^mv^n \rd v$. The 
other case is not possible since $n$ must be positive. 

\smallskip\noindent
{\bf Case 6.} \em Suppose $\te_1 (\bp) = 0$ and $\te_2 (\bp) = 0$. \em 
The point $\bp$ is a singularity of $\cF_1'$ and of $\cF_2'$ adapted to $\Sgm(\cF_1',\cF_2')$.
Up to permuting $u$ and $v$ we find $\te_1$ such that $\te_1 = v\rd u + u (\cdots) \rd v$. 
We know that $\te_2 = w\rd z + z (\cdots)\rd w$ for $(w,z) = (u,v)$ or $(w,z) = (v,u)$. 
We deduce, up to a multiplication by a local unit, we can find $\te_2$ 
such that $\te_2 = Unit\cdot\te_1 + u^mv^{n-1} \rd v$.
%
%
%
%
%
%
%
%
%
%
%
%
%
%
%
%
%
%
%
%
%
%
%
%
%
%
%
%
%
%
%
%
%
%
%
%
%
%
%
%
%
%
%
%
%
%

\section{Main result: Monomialization of $2$-symmetric tensors on regular surfaces}\label{section:main}
We present here the main result of the paper.  
At the very end of the section, we will recall the two situations 
we want to apply the main  result to, starting points of the paper.

\medskip
We recall some well known facts about morphisms between vector-bundles.

\smallskip
Let $\sgm:M\to N$ be a regular  mapping between regular manifolds.
Any fiber-bundle considered below will be a regular fiber-bundle, unless explicitly mentioned 
otherwise.

Let $F$ be a vector bundle of finite rank over $N$. The base change $\sgm: M\to N$ induces
a regular mapping of vector bundle $\sgm_F^h: F^\sgm \to F$, induces $\sgm$ on the $0$-section
and identity in the fibers. 

If $A :F \to F'$ is a regular mapping of vector bundles (both of finite rank) over $N$,
the base-change $\sgm$ induces a regular mapping $A^\sgm: F^\sgm \to (F')^\sgm$ of vector 
bundles over $M$.  

\medskip
Let $F$ be a vector bundle over $N$ and $E$ be a vector bundle over $M$, both of finite 
rank. Let $\Phi:E\to F$ be a \em regular vector bundles mapping along $\sgm$, \em that is such that 
$\pi^F_N\circ \Phi = \sgm \circ \pi^E_M$, where $\pi_\cB^\cE$ denotes the projection 
of the vector bundle $\cE$ onto its basis $\cB$.

The regular mapping $A\circ\Phi :E \to F'$ of vector bundles along $\sgm$ is thus well defined.

There exists also a unique regular mapping $\Phi^\sgm : E\to F^\sgm$ of regular vector bundles over $M$ 
factoring $\Phi$ through $\sgm_F^h$, namely $\sgm_F^h \circ \Phi^\sgm = \Phi$. 

\smallskip
The differential mapping $D\sgm:TM\to TN$ is a regular mapping of vector bundle along $\sgm$. Thus 
it factors as $D\sgm = \sgm_{TN}^h \circ \Delta\sgm$, where 
\begin{center}
$\Delta\sgm:= (D\sgm)^\sgm: TM \to TN^\sgm.$
\end{center}
This allows to pull-back any $\cO_M$-section $\te: M\to (TN^\vee)^\sgm$ as the regular 
$\cO_M$-section $\te\circ\Delta\sgm: M\to TM^\vee$, in other words, a regular $1$-form on $M$.

\medskip
Since we will pull-back differential forms and sub-modules of differential forms in the vector bundle (or module) 
sense as well as we will pull-back these differential forms and sub-modules in the sense of differential topology,
we repeat our choice of notations:

\smallskip\noindent
{\bf Important reminder about notations.} \em Let $\sgm: M\to N$ be a regular mapping.

\smallskip\noindent
 Let $\te$ be a regular differential $1$-form over $N$.
\\
- The notation $\te^\sgm$ will just mean $\te\circ\sgm$ section of $(TN^\vee)^\sgm$. 
\\
- The notation $\sgm^*\te$ will mean the pull-back 
of $\te$ in the usual sense of differential topology, that is 
$$
\sgm^*\te = \te\circ D\sgm \in \Omega_M^1 .
$$
The relation between these notations being: \em
\begin{center}
$\te\circ D\sgm = \sgm^* \te = (\te^\sgm)\circ\Delta\sgm = (\te\circ\sgm)\circ\Delta\sgm$.
\end{center}

\medskip

Let $S$ be a regular surface. 
Let $F$ be a regular vector bundle of rank $2$ over $S$ equipped with a fiber metric $\bg$. 
Let $\cL_F$ be a non-zero invertible $\cO_S$-sub-module of $\Gm_S (S(2,F))$.
Let $\tau:(T,E_T) \to S$ be a finite composition of point blowings-up so that $\cL := \cL_F^{\tau,div}$ satisfies
the conclusions of Proposition \ref{prop:param-reg}.
We denote $Y_\cL$ the snc-divisor $V_\cL \cup D_\cL \cup \Delta_\cL$ (see Section \ref{section:param-quad}). 
Either $\Delta_{\cL} := VD_{\cL}$ if $\cL_F$ is not constant along the fibers, or
$\Delta_{\cL} := V_\Te$ if $\cL_F$ is constant along the fibers (knowing then that $D_{\cL} = \emptyset$ 
once $\cL$ is not generically non-degenerate). 
\\
\indent
Let $\lambda: (T',E_{T'}) \to (T,\Lambda)$ be a locally finite sequence of point blowings-up, where $\Lambda = E_T \cup C$
for $C$ an isolated set of points of $T$. Therefore the exceptional divisor $E_{T'}$ decomposes as the union $E_T^\str \cup E_\lambda$, 
where \em $E_\lambda$ is the exceptional divisor of $\lambda$, \em that is $\lambda^{-1} (C)$.

\medskip
The next statement refers explicitly to Proposition \ref{prop:param-reg} and to notations introduced 
in Section \ref{section:param-quad}.
Any strict transform of a given nc-divisor $\Delta$ will be denoted by $\Delta^\str$, {\bf in the exception}
of exceptional divisors where strict transforms will still be denoted with the same symbol. 
\begin{theorem}\label{thm:main}
Let $S$ be a regular surface. 
Let $B$ be a regular vector bundle of rank $2$ over $S$ equipped with a fiber metric $\bg$. 
Assume that there exists a regular mapping $\cS: TS \to B$ 
of vector bundles over $S$ and a sub-variety $Z$ of $S$ of positive codimension, such that the restriction $\cS: TS|_{S\setminus Z} 
\to B|_{S\setminus Z}$ is a regular isomorphism of vector bundles over $S\setminus  Z$.

\medskip
Let $\cL_B$ be a non-zero invertible $\cO_S$-sub-module of sections of $(S(2,B))$.
Let be a finite composition of point blowings-up $\sgm_R:(R,E_R) \to S$ 
such that $\cL := (\cL_B\circ\cS)_R$ satisfies the conclusions of Proposition \ref{prop:param-reg}.
Then we find

\smallskip
1) There exists $\beta': (S',E') \to (R,E_R)$, a locally finite composition of point blowings-up,
where $E' :=E_{\beta'} \cup E_R$, 
such that the sub-variety $Z' = E' \cup {\sgm'}^{-1}(Z)$ is a snc-divisor normal crossing with 
$Y_\cL^\str$, where $\sgm':= \sgm_R\circ\beta'$.

\medskip
2) There exists $\wtbt: (\wtS,\wtE)\to (S',E')$, a locally finite sequence of point blowings-up, 
with $\wtE:=E_\wtbt\cup E'$,  
and let $\wtsgm$ be the composed mapping $\sgm'\circ\wtbt$,
such that 

\smallskip
i) At each point $\wtbp$ of $\wtS$ there exist a neighbourhood $\wtcU$ of $\wtbp$ in $\wtS$
and $1$-forms $\te_1,\te_2$ of $\Omg^1(\wtcU)$
such that each foliation generated by $\te_i$ has only simple singularities adapted 
to the snc divisor $\wtcU \cap \wtcD$, where $\wtcD$ denotes the snc-divisor $Y_\cL^\str\cup \wtZ$ 
for $\wtZ$ being $\wtE \cup \wtsgm^{-1}(Z)$.

\smallskip
ii) Each point $\wtbp \in \wtS$ admits a neighbourhood $\wtcU$ of $\wtbp$ in $\wtS$ such that, following-up
on the notations of i), denoting $\wtcS = \cS\circ D\wtsgm$ and $\kp$ a generator of $\cL_B$ nearby $\wtsgm(\wtbp)$, 
we find 
\begin{equation}\label{eq:good-param}
\kp\circ\wtcS =\cM_{\cL}[\ve_1\cN_1 (\cM_1 \te_1)\tns(\cM_1 \te_1) +  
\ve_2\cN_2 (\cM_2 \te_2)\tns(\cM_2 \te_2)], 
\end{equation}
where: $\ve_1,\ve_2 \in \{-1,1\}$; The germ $\cM_\cL$ is a monomial generating the ideal $\wtsgm^*\cC_\cL |_\wtcU$; 
For $i=1,2$, the germ $\cM_i$ is a monomial in $(\wtZ \cup \Delta_\cL^\str)\cap\wtcU$ such that 
the invertible sub-module $(\beta'\circ\wtbt)^*(\cQ_R\circ\Delta_{\sgm_R})$ of regular sections of 
$S(2,T\wtS)$ is generated over $\wtcU$ by $\cM_1\cM_2\cdot\te_1\cdot\te_2$, for $\cQ_R$ the invertible 
$\cO_R$-sheaf of regular sections of $S(2,(TS)^{\sgm_R})$ of 
Remark \ref{rmk:QR}.
\end{theorem}
Before getting into its demonstration, we complement Theorem \ref{thm:main} with the additional properties below, 
about the monomials appearing in Equation (\ref{eq:good-param}). They are more technical in nature and similar to those
stated in Proposition \ref{prop:param-reg}:
\begin{corollary}\label{cor:main}
We can further see that
\\

(iii) $\cO_\wtcU(\te_1 \wedge \te_2) = \cJ_{1,2} \cdot \Omega_\wtcU^2$, 
where $\cJ_{1,2}$ is a monomial and principal ideal in the snc-divisor 
$\wt{\Sgm}_{1,2}:=\cosupp (\cJ_{1,2}\cO_{\wtcU})$  (containing $(\wtZ\cup\Delta_\cL^{\str})\cap\wtcU$) 
which is normal crossing with $\wtcD\cap\wtcU$.
\\
(iv) Let $\cM_{1,2}^{\log}$ be a local generator of $\cJ_{1,2}^{\log}$ at $\wtbp$. The 
ideal of logarithmic coefficients of $\cO_\wtcU(\te_1\wedge\te_2)$ as a sub-module of 
$\Omg_{\wtS}^2(\log(\wtcD \cup \wt{\Sgm}_{1,2}))$. For $i=1,2,$ let 
$\cM_i^{\log}$ be a local generator of the logarithmic coefficient ideal $\cC_i^{\log}$
of $\te_i$ as a sub-module of $\Omg_{\wtS}^1(\log(\wtcD \cup \wt{\Sgm}_{1,2}))$.
The monomials $(\cM_\cL\cN_1\cM_1^{\log})$, $(\cM_\cL\cN_2\cM_2^{\log})$,
$\cM_{1,2}^{\log}$ are ordered.

\smallskip\noindent
3) We can track the monomials and their vanishing locus:
\begin{itemize}
\item if $\cL$ is generically non-degenerate, both function germs $\cN_1,\cN_2$ are 
monomials in $(\wtZ \cup D_\cL^\str)\cap\wtcU$ which cannot both vanish 
simultaneously. And one of the monomial $\cN_i\cM_\cL$ is a local 
generator of the ideal $\wtsgm^* I_\cL^D|_\wtcU$.
\item if $\cL$ is everywhere degenerate one of the function germs $\cN_1,\cN_2$ is a local 
monomial in $\wtZ \cup D_\cL^\str$ while the other one is identically zero. 
If $\cN_i$, for $i=1$ or $2$, is not the zero monomial then $\cN_i\cM_\cL$ 
is a local generator of the ideal $\wtsgm^* I_\cL^D|_\wtcU$.
\end{itemize}
\end{corollary}
We now proof both the Theorem and its Corollary.
\begin{proof}
Point 1) is straightforward. 

\medskip
Let $\cQ_R$ be the invertible $\cO_R$-sub-module of regular sections of $S(2,TR)$ of Remark 
\ref{rmk:QR}, and let $\cQ' := \cQ_R\circ \sgm'$ encoding the diagonalizing directions of $\cL^{\sgm'}$.
We define  $\cQ^\bullet := \cQ_R\circ\Delta_{\sgm_R}\circ D\beta' = \cQ' \circ \Delta_{\sgm'}$, 
an $\cO_{S'}$-invertible sheaf of sections of $S(2,TS')$.

\medskip
Let $E_{\beta'}$ be the exceptional divisor produced by $\beta'$. Let $\cup_{i\in I'} \{\bp_i\}$ be 
the image $\beta'(E_{\beta'})$. Since these points are isolated, we can assume for the work to come that 
there is just a single one, this will not change our arguments since there are local in nature.

\smallskip\noindent
{\bf Special case:} \em $B$ is regularly isomorphic to $TS$, in other words $Z = \emptyset$. \em 

\smallskip
Let $\bp_R$ be the point $\beta'(E_{\beta'})$. 
To ease slightly the notations, let $B_R := B^{\sgm_R}$.
By the results of Section \ref{section:param-quad}, there exist an open neighbourhood $\cU_R$ 
of $\bp_R$ and two regular sections of $B_R^\vee|_{\cU_R}$, say $\eta_1,\eta_2$, vanishing 
nowhere in $\cU_R$ with orthogonal kernels and such that $\eta_1\cdot\eta_2$ generates $\cQ_R$ at every
point of $\cU_R$. Moreover, up to shrinking $\cU_R$, we find 
$$
\kp_R|_{\cU_R} :=\kp^{\sgm_R}|_{\cU_R} = a_1\eta_1\otimes\eta_1 + a_2 \eta_2\otimes\eta_2 
$$
for regular functions $a_1,a_2$ over $\cU_R$ monomial in some snc divisors.
Outside this neighbourhood all the results of 2) are trivially true, since $\beta'$ induces a regular isomorphism.
In particular at every point of $\cU_R$ we find that $\cQ^\bullet$ is generated by $\beta'^*\eta_1\cdot\beta'^*\eta_2$.

Let $\cU' := \beta'^{-1}(\cU_R)$, a neighbourhood of $E_{\beta'}$, over which are defined two regular forms
\begin{center}
$\omg_i:=\eta_i\circ(\cS^{\sgm'}\circ\Delta\beta')$ for $i=1,2$. 
\end{center}
In particular we have two singular foliations over $\cU'$, which are orthogonal outside $E_{\beta'}$, and have singularities
along $E_{\beta'}$.

Let $\cD_i$ be the sub-module of $\Omg_{S'}^1|_{\cU'}$ generated by $\omg_i$ for each $i=1,2$. 
Since $\omg_i$ is singular along $E_{\beta'}$ we can factor out a principal $\cO_{\cU'}$-ideal
$\cC_{\cD_i}$ monomial in $E_{\beta'}$ such that 
$$
\cD_i = \cC_{\cD_i}\cdot\cD_i'
$$
where $\cD_i'$ is an invertible $\cO_{\cU'}$-sheaf of regular sections of $\Omg_{\cU'}^1$ with isolated support, if not empty.
More precisely if $\cC_{\cD_i}$ is locally generated by $h_i$, then $\cD_i'$ is generated locally by $\omg_i^r :=(h_i)^{-1}\omg_i$.
We have actually more than that: The whole invertible sheaf $\cQ^\bullet$ decomposes
as $\cI_\bullet\cdot\cQ^{\bullet,r}$ where $\cI_\bullet$ is principal and monomial in $E_{\beta'}$ and $\cQ^{\bullet,r}$ 
has co-support of dimension $2$ or is empty.

\smallskip\noindent
{\bf Observation.} Let $\gm:(S'',E''= E'\cup E_\gm) \to (S',E')$, be the blowing-up of the point $\bp'\in \cU'$ 
and $E_\gm$ be $\gm^{-1}(\bp')$, the newly created exceptional hypersurface. Let $I_{E_\gm}$ be the reduced ideal 
of $E_\gm$. Let $\cU''$ the pre-image $\gm^{-1}(\cU')$.  We observe that $\gm^*\cD_i' = I_{E_\gm}^{-k_i}|_{\cU''} \cD_i''$, 
where $k_i$ is a positive integer and $\cD_i''$ is a sub-module of $\Omg_{\cU''}^1$ which is finite co-dimensional at each point. 

\smallskip
The simple observation above guarantees that there exists a locally finite sequence 
of points blowings-up $\beta'': (S'',E''= E'\cup E_{\beta''}) \to (S',E')$ such that for each $i$
the sub-module $\beta''|_{\cU''}^* \cD_i$ factors as $J_i|_{\cU''}\cdot\cD_i''$,
 where $J_i$ is a principal ideal monomial in $E_{\beta''}$, and $\cD_i''$ is an invertible 
$\cO_{\cU''}$-sub-module of $\Omg_{\cU''}^1$, of finite co-dimension at each point,
where $\cU''$ is the pre-image $\beta''^{-1}(\cU')$.
 
\smallskip
In order to avoid further notations, we can assume that $\sgm'$ is already such that each 
ideal $\cC_{\cD_i}$ is already principal and monomial in $E'$, so that each $\cD_i'$ 
is also already defining a foliation $\cF_i'$ on $\cU'$ for $i=1,2$.

\smallskip
Now, we have to resolve the singularities of $\cF_1'$ and $\cF_2'$ and do further point
blowings-up so that each final pulled-back foliation is in a form as good as it can be with some of
the given snc-divisors we want to take care of. But, up to a locally finite sequence of points blowings-up we can 
already assumed, thanks to the results of Section \ref{section:pairs-foliat}, that the mapping $\sgm'$ achieve this.
So we get point i). To get the whole of point ii) there is just to carefully track everything 
we have at the level of Equation (\ref{eq:param-sum-squares}) of Proposition \ref{prop:param-reg}
for $\kp_R$ and since 
$$
\kp\circ\cS\circ D\wtsgm = (\kp_R\circ\cS\circ(\sgm_R)_B^h\circ\Delta\sgm_R\circ D\zeta), 
$$
where $\zeta$ is defined as $\wtsgm = \zeta\circ\sgm_R$, we check we get what 
is stated (since we have assumed $\cS$ is an isomorphism).

\smallskip
Now we deal with point iii). Let $\Delta'$ be the snc-divisor 
$V_{\cL_\cS}^\str\cup D_{\cL_\cS}^\str \cup \Delta_{\cL_\cS}^\str$. Let 
$\cJ_{1,2}'$ be the coefficients ideal of $\sgm'^*(\Omg_S^2) = \cJ_{1,2}'\cdot \Omega_{S'}^2$.
We also find 
$$
\cD_1\wedge\cD_2 =  \cC_{\cD_1}  \cC_{\cD_2}J_{1,2} \cdot \Omega_{S'}^2
$$
with $\cC_{\cD_1}\cC_{\cD_2}$ principal and monomial in the snc divisor $E_{\beta'}$ while $J_{1,2}$ has 
co-support $\Sgm'$ which does not contain any component of $E_{\beta'}$.
In particular the sub-variety $\Sgm'$ is  
the tangency locus, contained in $\cU'$ of the foliations $\cF_1'$ and $\cF_2'$. We can 
assume, up to further point blowings-up, that $\Lambda'$, the co-support of $(\cJ_{1,2}')$, and $\Sgm'$ 
are snc-divisors which are normal crossings with $\Delta'$ and that $\cJ_{1,2}'$ is also principal 
and monomial in $\Lambda'$. Note that we also have $\Lambda^\prime \subset \Sigma^\prime \cup E_{\beta^\prime}$.
We can assume moreover, up to further point blowings-up, 
that each local component of $\cE':=\Sgm' \cup \Delta'\cup E^\prime$, contained in $\cU'$,
is either invariant or di-critical for both foliation $\cF_i'$, $i=1,2$.

At a regular point of $\cE'$, each monomial in $\cE'$ under scrutiny is of the form $u^l$ for $u$ a local 
coordinate and $l$ a non-negative integer. So they are already ordered. Let $X'$ be the subset of 
corner points of $\cE'$. Thus at each point $\bq'$ of $X'$ and for each $i=1,2$, each local component 
of $\cE'$ is either invariant or di-critical for $\cF_i'$. This fact is important since the proof of 
point v) and point vi) of Proposition \ref{prop:pairs-order} shows that we can always order 
the "logarithmic" monomials $\cM_1^{\log}$ and $\cM_2^{\log}$. Thus,
working with the logarithmic $1$-forms along the snc-divisor $\cE'$, 
these logarithmic monomials can be assumed already ordered
at any corner point of $\cE'$.

Let us repeat the argument here: Let $\mu_i$ be a local logarithmic generator of $\cD_i$
so that the pull-back $\sgm'^*\mu_i = \cM_i^{\log}\te_i^{\log}$, where $\te_i^{\log}$ is a 
local logarithmic generator of $\cF_i'$ and where $\sgm'^*\cD_i$ is seen as a sub-module of 
$\Omg_{\cU'}^1(\log \cE')$. If $\gm$ is the blowing-up in $S'$ of the point $\bq'$ of $E'$, we 
see that at each corner point of $\gm^{-1}(\cE')\cap\gm^{-1}(\bq')$, we find out
that $\gm^*\te_i^{\log}$ is indeed a logarithmic generator of the pulled-back foliation  
$\gm^*\cF_i'$, so that a local generator of the ideal of logarithmic coefficients 
of $(\sgm'\circ\gm)^*\cD_i$ is just the pull back by $\gm$ of a local generator of the 
ideal of logarithmic coefficients of $\sgm'^*\cD_i$. Our problem of comparison of monomials 
is indeed just a problem of comparing monomials, forgetting about the foliations.
 
Thus at such a point $\bq'$ of $E'$, there exists $\pi: (S'',E'') \to (S',E')$, a finite sequence of point
blowings-up, such that the pull-back of the monomials, we were looking for to order at $\bp'$, are 
ordered at each corner point of $\pi^{-1}(\bq') \cap \pi^{-1}(\cE')$.

\medskip\noindent
{\bf General case:} \em Assume that $Z$ is not empty. \em 

\smallskip
In this context $TS$ play the role of $B$ of the special case above and with $\cL = \cL_B\circ\cS$
in the stead of $\cL_B$ of the special case.
What needs to be done in order to get the announced results, is to track down the position of the strict transform 
$Z^\str$ of $Z$ with the data we already have, and by means of further additional points blowings-up, put the 
new corresponding strict transform  in general position 
with the pull back of the data we already have. This can be achieved by the results of 
Section \ref{section:pairs-foliat}.

\medskip
Point 3) requires just to keep track of the whole process of the monomialization of the corresponding 
ideal, initiated in Section \ref{section:param-quad}.
\end{proof}
The \`a-priori artificial context of Theorem \ref{thm:main} 
proceeds from finding a formulation for the two, similar but not identical, 
following situations below. Instead of working only with 
the tangent bundle, we work on any regular vector bundle of rank $2$. 
Nevertheless, due to a cell decomposition of $S$, there
exists always a closed subset $F$ of $S$, complement of the cells of dimension $2$ 
such that $TS|_{S\setminus F}$ is isomorphic to $B|_{S\setminus F}$, since the connected components
$S\setminus F$ are open balls, thus contractible. Whether $F$ can be a sub-variety 
deserves to be discussed.

\smallskip
The first situation is when $S$ is a regular surface and $B =TS$ its tangent bundle. 
Thus $\kp$ can be any $2$-symmetric (regular) tensor (field) on $S$, and may be degenerate 
somewhere (see \cite{Gra1,GrSa} for semi-positive definite examples).

\smallskip
The second situation motivated the hypotheses on $Z$, $B$ and $TS$ in the theorem.
Suppose the regular surface $S$ resolves the singularities of a surface singularity $S_0$ 
embedded in a regular manifold $M_0$, such that it factors through an embedded resolution 
of the singularities of $S_0$ such that the resolution mapping $\sgm:S \to S_0$ is Gauss-regular, with exceptional divisor $E$,
which is possible by Proposition \ref{prop:gauss-regular} (see Appendix \ref{section:appendixA}). 
Taking $B := T^\sgm S_0$ (See Appendix \ref{section:appendixA}), outside the exceptional divisor $E$ 
the mapping $\sgm$ induces an isomorphism between $TS$ and $B$.
We take $\cL$ as generated by the pull-back of $\sgm^* (\cK|_{S_0})$ of any given 
invertible $\cO_{M_0}$-sub-module $\cK$ of $\Gm_{M_0} (S(2,TM_0))$. 
As explained in the introduction, we came across such situations when $\cK$ is generated by a given 
regular metric on $M_0$ \cite{GrGr,GrSa}.
\begin{remark}
The result proved above does not depend on the Riemannian metric $\bg_0$ but only 
on its conformal class, in other word depends only on the invertible $\cO_{M_0}$-sub-module
of $\Gm_{M_0} (S(2,TM_0))$ generated by $\bg_0$. Indeed, the choice of the geometrically 
admissible centers we blow-up (to reach our main result) is not affected at any step, if instead of 
working with $\bg_0$ we were working with a conformal metric, since the only feature of $\bg_0$ we 
really need to keep track at any time is simply the notion of orthogonality.
\end{remark}
%
%
%
%
%
%
%
%
%
%
%
%
%
%
%
%
%
%
%
%
%
%
%
%
%
%
%
%
%
%
%
%
%
%
%
%
%
%
%
%
%
%
%

\bigskip
\appendix
\section{Gauss Regular Resolution and 2-tensors on singular sub-varieties}\label{section:appendixA}
We present here further ingredients related to resolution of singularities of sub-varieties.
The estranged formulation of the main result, Theorem \ref{thm:main}, finds some justifications 
in the complements of this section.
\subsection{Resolution of singularities with Gauss regular mapping}\label{subsection:appendixA1}
The material presented here, although part of the known folklore, introduces useful
notions and notations. We are very grateful to Pierre Milman for telling us about Gauss regular 
desingularization.

\medskip
Let $\bG_k (V)$ be the Grassmann-bundle of $k$-dimensional real vector subspaces of the 
finite dimensional real vector space $V$. Let $[P]$ be the point of $\bG_k (V)$ corresponding 
to the $k$-dimensional vector subspace $P$ of $V$. 

\smallskip  
Let $F$ be a regular vector bundle of positive finite rank $r$ over a regular (connected) manifold $N$ of finite dimension. 
Let $\bG_k (F)$ be the Grassmann bundle  of the $k$-vector subspaces in the fibers of $F$.
Let $\bG(F)$ be $\cup_{k=1}^r \bG_k(F)$ the total Grassmann bundle of $F$. 

\medskip
Let $M_0$ be a connected regular manifold of dimension $n$. 
Let $X_0$ be a singular sub-variety of the regular manifold  $M_0$. Let $Y_0$ be the singular 
locus of $X_0$. 
The Gauss mapping of $X_0$ is defined as 
$$
\begin{array}{rccl}
\nu_{X_0}  \; : \; &   X_0\setminus Y_0 \; & \; \to \; & \bG(TM_0)  \\
&  \ulb_0 \in X_0\setminus Y_0 & \; \to \;  & \nu_{X_0}(\ulb_0) = [T_{\ulb_0} X_0] \in \bG_{\dim (X_0,\ulb_0)}(T_{\ulb_0} M_0).
\end{array}
$$
\begin{definition}\label{def:gauss-regular}
A geometrically admissible resolution of singularities $\pi: (X,E) \to (X_0,Y_0) \hookrightarrow M_0$ of $X_0$
is said \em Gauss regular, \em if the mapping $\nu_{X_0} \circ \pi$ extends over $X$ 
as a regular mapping $X \to \bG(TM_0)$.
\end{definition}
Composing a geometrically admissible Gauss regular resolution of singularities of $X_0$ with any 
geometrically admissible blowing-up with center in the exceptional divisor will yield another Gauss 
regular resolution of singularities of $X_0$.
\begin{proposition}[see \cite{BBGM}]\label{prop:gauss-regular}
There exists a Gauss regular resolution of singularities of $X_0$.
\end{proposition}
\begin{proof}
For simplicity we suppose that $X_0$ is of pure dimension $d$.

\smallskip
Let $\tau_1:(M_1,X_1,E_{M_1}) \to (M_0,X_0,Y_0)$ be a geometrically admissible embedded resolution 
of singularities of $X_1$. Let $\sgm_1$ be the restriction mapping $\tau_1|_{X_1}$ and let 
$E_1$ be the intersection $X_1 \cap E_{M_1}$ which is a snc-divisor of the resolved manifold $X_1$.

\smallskip
Let $F_0(\sgm_1)$ be the $\cO_{X_1}$-ideal sheaf locally generated by the 
maximal minors of the differential mapping $D\sgm_1$ whose
co-support is the critical locus of $\sgm_1$, contained in $E_1$.
Given any geometrically admissible blowing-up $\beta_C$ with center $C$ contained in $X_1$, 
there exists a non-negative integer $\alpha$ (depending on $C$)
such that the ideal $F_0(\sgm_1\circ\beta_C)$ factors as $(I_{E_C})^{\alpha} \cdot F_0(\sgm_1)^{\beta_C}$, 
where $E_C$ is the newly created exceptional hypersurface $\beta_C^{-1}(C)$ and $I_{E_C}$ is 
its reduced (and principal) ideal. Up to further geometrically admissible 
blowings-up (with centers in $E_1$), we can assume that $F_0(\sgm_1)$ is already
principal and monomial in $E_1$.

For any point $\ula_1$ of $X_1\setminus E_1$, we know that $D\sgm_1 (\ula_1)\cdot T_{\ula_1}X_1 = 
T_{\sgm_1 (\ula_1)} X_0$. 
Let $\ula_1$ be any point of $E_1$ and let $(u,v)$ be local coordinates adapted to $E_1$.
Let $(u',v')$ be another system of local coordinates adapted to $E_1$. 
Thus, in a neighbourhood $\cU_1$ of $a_1$ in $X_1$,
\begin{center}
$D\sgm_1 \cdot \dd_{u_1} \wedge \cdots \wedge D\sgm_1 \cdot\dd_{u_s} \wedge 
D\sgm_1 \cdot \dd_{v_1} \wedge \cdots \wedge D\sgm_1 \cdot \dd_{v_t} = \hfill$ 

$\hfill Unit \cdot
D\sgm_1 \cdot \dd_{u_1'} \wedge \cdots \wedge D\sgm_1 \cdot \dd_{u_s'} \wedge 
D\sgm_1 \cdot \dd_{v_1'} \wedge \cdots \wedge D\sgm_1 \cdot \dd_{v_t'}$
\end{center} 
where $s+t =d$ is the dimension of $X_0$.

\smallskip
Since the ideal $F_0(\sgm_1)$ is principal and monomial in $E_1$, there exists 
a nowhere vanishing regular mapping $\gm_1:\cU_1 \to \wedge^d TM_0$ such that 
\begin{center}
$D\sgm_1 \cdot \dd_{u_1} \wedge \cdots \wedge D\sgm_1 \cdot \dd_{u_s} \wedge 
D\sgm_1 \cdot \dd_{v_1} \wedge \cdots \wedge D\sgm_1 \cdot \dd_{v_t} = \cM_1 \cdot \gm_1$
\end{center}
where $\cM_1$ is a local generator of $F_0(\sgm_1)$.
So we deduce that the mapping 
\begin{center}
$\cU_1 \ni \ula \to [\gm_1 (\ula)] \in \bG_d(TM_0)$
\end{center} 
(using here the Pl\"ucker  embedding of $\bG_d(TM_0)|_{\cU_0}$, where $\cU_0$ is a neighbourhood of $\sgm_1(\ula_1)$ over 
which $TM_0$ is trivial) where $[\gm_1 (\ula)]$ is the vector space direction corresponding to 
the $d$-vector $\gm_1 (\ula)$. This regular mapping coincides with $\nu_{X_0} \circ\sgm_1$ on $\cU_1 
\setminus E_1$.

\medskip
When $X_0$ is not of pure dimension, the resolved manifold $X_1$ is a disjoint union 
of regular manifolds and proceed exactly as above, independently for each dimension.
\end{proof}
%
%
%
%
%
%
%
%
%
%
%
%
%
%
%
%
%
%
%
%
%
%
%
%
%
%
%
%
%
%
%
%
%
%
%
%
%
%
%
%
%
%
%
%
\subsection{$2$-symmetric tensors and quadratic forms on singular sub-varieties}\label{subsection:appendixA2}
Using Subsection \ref{subsection:appendixA1} we make sense 
here of the notion of restriction of $2$-symmetric tensor, respectively quadratic forms, 
on singular sub-varieties. 

\medskip
Let $V$ be a real vector space of finite dimension $r$. 
The \em universal Grassmann bundle $\wtbG_k (V)$ \em is the algebraic sub-variety 
of $\bG_k(V) \times V$ consisting of the pairs $([P],v)$ for any vector $v\in P$ with 
$P$ any $k$-dimensional vector subspace of $V$. It is also an algebraic real vector bundle 
of rank $k$ over $\bG_k(V)$.

\medskip 
Let $M_0$ be a connected analytic manifold of finite dimension. 

\smallskip
Let $Z$ be any non-empty  sub-variety of $M_0$.
The \em Nash "bundle" $\cN_Z$ of $Z$ in $M_0$ \em is the closed semi-analytic subset of $\bG(TM_0)$ 
obtained as the (topological) closure of the graph of the Gauss mapping of $Z$. 
Let $C_4(Z)$ be the closure, taken into $TM_0$, of the tangent bundle $T Z_\reg$ of the regular part 
$Z_\reg$ of $Z$. We call it 
the \em pseudo-tangent "bundle" of $Z$. \em We denote it by $C_4(Z)$, since 
point-wise, the fiber $C_4(Z,\ula)$ over a point $\ula$ of $X_0$ is the fourth Whitney tangent cone 
\cite{Wh1} and consists of the union $\cup_{[P]\in \cN_\ula Z} P$ of all 
the limits at $\ula$ of the tangent spaces to $Z$ at regular points of $Z$. 
When $Z$ is a sub-manifold $C_4(Z)$ is just the usual tangent bundle $TZ$. 

\medskip
Let $X_0$ be a singular sub-variety of $M_0$ with non empty singular locus $Y_0$.
Since $C_4(X_0)$ is a subset of $TM_0$ we can introduce the following
\begin{definition}\label{def:restriction}
Let $\kappa$ be a regular quadratic form on $M_0$. The \em restriction of $\kp$ to $X_0$,
\em denoted $\kp|_{X_0}$, is defined as the restriction $\kp|_{C_4(X_0)}$ of $\kp$ to the 
pseudo-tangent "bundle" of $X_0$. 

The restriction of the $2$-symmetric tensor $\cK$ on $M_0$ to $X_0$ is just defined via the polar form of 
the restriction of $\cK_\Dlt$ to $X_0$. 
\end{definition}

\medskip
Let $\tau: (T_0,D_0) \to X_0$ be any Gauss regular admissible resolution of singularities of $X_0$. 
Let $\nu_0$ be the regularized Gauss mapping of $X_0$, that is the regular mapping $T_0 \to \bG(TM_0)$ 
extending to the whole of $T_0$ the parameterized Gauss mapping $\nu_{X_0}\circ \tau: T_0\setminus D_0 
\to X_0\setminus Y_0$. We see that 
$$
\cup_{\ula\in X_0} \cup_{\ulb\in\tau^{-1} (\ula)}(\ula,\nu_0 (\ulb)) = \cN_{X_0}.
$$
For any point $\ulb$ in $T_0$, let $T_\ulb^\tau X_0$ be the vector subspace of $T_{\tau(\ulb)} M_0$ 
whose direction is the value at $\ulb$ of the regular extension $\nu_0$, namely 
$\nu_0(\ulb) = [T_\ulb^\tau X_0]\in \bG (T_{\tau(\ulb)} M_0)$. 
We call the vector sub-space $T_\ulb^\tau X_0$ \em the tangent space of $X_0$ at $\ulb$ along $\tau$. \em
We deduce that 
$$
C_4(X_0) = \cup_{\ula\in X_0} \cup_{\ulb\in\tau^{-1} (\ula)} \ula\times T_\ulb^\tau X_0
$$
and that for each point $\ulb$ of $T_0$, the differential mapping $(D\tau)(\ulb) :T_\ulb T_0 \to T_{\tau (\ulb)} M_0$ 
takes its values in $T_\ulb^\tau X_0$. 

Let $\wtbG_k (TM_0)$ be the  universal bundle associated with $\bG_k (TM_0)$ and 
let 
\begin{center}
$\wtbG (TM_0):= \cup_{k=1}^r \wtbG_k (TM_0)$,
\end{center}
be the corresponding universal bundle. 
Let $\wtpi: \wtbG (F) \to F$, defined as $(\ula,[P],v) \to (\ula,v)$.

Taking the graph of $\nu_0$, embedding it in the fibered product $T_0\times_{M_0} \bG (TM_0)$,
then lifting it in the fibered product $T_0\times_{M_0} \wtbG (TM_0)$ and eventually 
projecting this lift in the fibered product $T_0\times_{M_0} TM_0$
via the mapping $\wtpi$ shows that the union 
$$
T^\tau X_0:=\cup_{\ulb\in T_0} T_\ulb^\tau X_0,
$$
called \em the tangent bundle of $X_0$ along $\tau$, \em 
is a regular vector bundle over the resolved manifold $T_0$. 
Outside the critical locus of $\tau$ the restricted vector bundle $(T^\tau X_0)|_{T_0\setminus D_0}$
is just the pull-back $T (X_0\setminus Y_0)^{\tau|_{T_0\setminus D_0}}$.

\medskip
Thanks to Definition \ref{def:restriction}, the restriction of 
any submodule of $\Gm_{M_0} (S(2,TM_0))$ to $X_0$ is well defined.

Suppose given $\pi_1: (X_1,E_1) \to (X_0,Y_0)$, a Gauss regular 
resolution of singularities of $X_0$ and let $\cL_0$ be an invertible sub-module of $\Gm_{M_0} (S(2,TM_0))$. Thus 
the regular "section" $\cL_0^{\pi_1}|_{T^{\pi_1}X_0}$ 
of $S(2,T^{\pi_1} X_0)$ coincides with $(\cL_0|_{X_0})^{\pi_1}$. 
Namely for $\kp$ a local generator of $\cL$ nearby $\ula_0$ in $X_0$ and for any $\ula_1$ in $\pi_1^{-1}(\ula_0)$, 
we find
$$
(\kp^{\pi_1}|_{T^{\pi_1} X_0} )(\ula_1) = \kp(\pi_1(\ula_1))|_{T_{\ula_1}^{\pi_1}X_0}.
$$

\smallskip
Let $\cL_1:=\cL_0^{\pi_1}|_{T^{\pi_1}X_0}$ and let $\cC_{\cL_1}$ be the $\cO_{X_1}$-ideal
of coefficients of $\cL_1$ obtained by
evaluating the "$2$-symmetric tensor" $\cL_1$ along the regular section germs of 
$S(2,T^{\pi_1} X_0)$.
With further locally finite geometrically admissible blowings-up we can assume that 
$\cC_{\cL_1}$ is principal and monomial in $E_1$.
Thus any local generator of the invertible $\cO_{X_1}$-submodule $\cC_{\cL_1}^{-1}\cL_1$
does not vanish anywhere. 
%
%
%
%
%
%
%
%
%
%
%
%
%
%
%
%
%
%
%
%
%
%
%
%
%
%
%
%
%
\section{Local normal forms of differentials and of the inner metric on singular surfaces}\label{section:normal-form}
We complete the paper addressing the primary motivation of this work: describing 
locally, in a resolved manifold, the pull-back of the inner metric, by the resolution mapping, of an 
embedded real surface singularity. As a consequence of the previous sections we get a proof of the Hsiang \& Pati 
property for real surfaces which is a bit different from the existing ones \cite{HP,PaSt,Gri1,BBGM}. 
\subsection{Hsiang \& Pati property}\label{subsection:hsiang-pati}
Let $M$ be a smooth manifold. Two (Riemannian) metrics $\bg$ and $\bh$ on $M$
are \em quasi-isometric \em if there exists a positive 
constant $C$ such that $C^{-1}\bh \leq  \bg \leq C \bh$.
They are \em locally quasi-isometric \em if each point $\ula$ of $M$ admits
an open neighbourhood $\cU$ such that the restricted metrics $\bg|_\cU$ and 
$\bh|_\cU$ are quasi-isometric.  

\smallskip
The next result, of local nature, is the main tool used by Hsiang \& Pati to get their result. 
\begin{lemma}[{\cite[Section III]{HP}}]\label{lem:HPcoord}
Let $(X_0,\bo)$ be a normal complex isolated surface singularity germ embedded in $(\C^N,\bo)$.
There exists a finite composition of points blowings-up $\sgm: (X,E) \to (X_0,\bo)$
such that:

i) $X$ is a complex manifold of dimension two and $E:=\sgm^{-1}(\bo)$, the exceptional divisor
of this desingularization of $(X_0,\bo)$, is a snc-divisor.

ii) Any regular point $\ula$ of $E$ admits local regular coordinates $(u,v)$, centered
at $\ula$, such that in this chart $(E,\ula) = \{u=0\}$ and 
the resolution mapping writes locally 
\begin{multline}\label{eq:HP-smooth}
(u,v)  \to (x,y;z) = \sgm(\ula) + (u^{r+1},u^{r+1}f(u) + u^{r+s+1}v; u^{r+1}g(u) + u^{r+s+1}Z(u,v)) \\
\in \C\times\C\times\C^{N-2}
\end{multline}
for non-negative integers $r,s$, for germs $f\in \C\{u\}$, $g\in \C\{u\}^{N-2}$, and $Z$ a regular map germ $(X,\ula) \to \C^{N-2}$.

iii) Any corner point $\ula$ of $E$ admits local regular coordinates $(u,v)$, centered
at $\ula$, such that in this chart $(E,\ula) = \{uv=0\}$ and 
the resolution mapping writes locally 
\begin{multline}\label{eq:HP-corner}
(u,v)  \to (x,y;z) = \sgm (\ula)+(u^mv^n, u^mv^nf(u,v) + u^pv^q; u^mv^ng(u,v) + u^pv^q Z(u,v))  \\ 
\in \C\times\C\times\C^{N-2}
\end{multline}
for non-negative integers $p \geq m$ and $q\geq n$ such that $np -qm \neq 0$ and germs $f \in \cO_\ula$,
$,g \in \cO_\ula^{N-2}$ and
$Z\in \cO_\ula^{N-2}$ such that 
$\rd f \wedge \rd(u^mv^n) = \rd g \wedge \rd(u^mv^n) \equiv 0$.
\end{lemma}
Such local coordinates $(u,v)$, in either cases, are called \em Hsiang \& Pati coordinates. \em
The corollary of such systematic local presentation of the resolution mapping is 
Hsiang \& Pati result of interest to us: 
\begin{theorem}[\cite{HP}]
Let $X_0$ be a normal complex surface singularity germ embedded in $\C\field{P}^N$.
Let $\bg_{X_0}$ be the restriction to the regular part of $X_0$ of the Fubini-Study 
metric on $\C\field{P}^N$.
There exists a finite composition of points blowings-up $\sgm: (X,E) \to X_0$
resolving the singularities of $X_0$ such that 

i) Each point $\ula$ of $E$ admits Hsiang \& Pati coordinates $(u,v)$ like in Lemma \ref{lem:HPcoord}.

ii) When $\ula$ is a regular point of $E$, the (regular extension of the) pulled-back metric 
$\sgm^*\bg_{X_0}$ is quasi-isometric to the metric over $\cU$ given by 
\begin{center}
$\rd u^{r+1} \tns\overline{\rd u^{r+1}} + \rd u^{r+s+1}v \tns \overline{\rd u^{r+s+1}v}$.
\end{center}

iii) When $\ula$ is a corner point of $E$,  the (regular extension of the) pulled-back metric 
$\sgm^*\bg_{X_0}$ is quasi-isometric to the metric over $\cU$ given by 
\begin{center}
$\rd u^mv^n \tns\overline{\rd u^mv^n} + \rd u^pv^q \tns\overline{\rd u^pv^q}$.
\end{center}
\end{theorem}
\subsection{Preliminaries for local normal forms}\label{subsection:prel-form}
$ $ 

\smallskip
Let $M_0$ be a regular connected manifold of dimension $N$, equipped with a regular Riemannian 
metric $\bg_0$. Let $X_0$ be a sub-variety with no connected component of dimension other than $2$. 

\smallskip
Suppose given a Gauss regular resolution $\sgm_1:(X_1,E_1)\to X_0$.

\smallskip\noindent
{\bf Notation.} 
\em Let $\Omg_{\sgm_1}^1$ be the $\cO_{X_1}$-dual to $\Gm_{X_1} (T^{\sgm_1}X_0)$. \em
It is a locally free $\cO_{X_1}$-module of rank $2$

\smallskip
A \em differential $1$-form along $\sgm_1$ \em is a regular section $X_1 \to (T^{\sgm_1} X_0)^\vee$.

\smallskip
Following Proposition \ref{prop:gauss-regular} and then using Proposition \ref{prop:param-reg}, 
we can further assume that for any point $\ula_1$ in $X_1$ there exist a neighbourhood $\cU_1$ of 
$\ula_1$ and local regular sections $\omg_1,\omg_2$ of $\Omg_{\sgm_1}^1|_{\cU_1}$,
with orthogonal kernels (for the fiber-metric $\bg^{\sgm_1}$ restriction of $\bg_0\circ\sgm_1$ 
to $T^{\sgm_1}X_0$), such that over $\cU_1$ the following holds true:
\begin{equation}\label{eq:metric-param-1}
\bg^{\sgm_1} = 
\omg_1\tns\omg_1 + 
\omg_2\tns\omg_2.
\end{equation}
By definition $\bg^{\sgm_1}$ is also the fiber metric onto $T^{\sgm_1} X_0$ which extends 
the restriction of the fiber-metric $(\bg_0|_{X_0\setminus Y_0})\circ\sgm_1$ to $T^{\sgm_1}X_0|_{X_1\setminus E_1} 
= T(X_0\setminus Y_0)^{\sgm_1|_{X_1\setminus Y_1}}$.

\smallskip
Suppose given a resolution of singularities $\wtpi:(\wtX,\wtE) \to X_0$ like in Theorem \ref{thm:main}
factoring through $\sgm_1$, that is $\wtpi = \sgm_1\circ\wtbt$ for $\wtbt:(\wtX,\wtE)\to (X_1,E_1)$ a locally
finite sequence of point blowings-up. Thus Equation (\ref{eq:metric-param-1}) becomes
\begin{equation}\label{eq:metric-param-2}
\bg^{\wtpi}:= \te_1\tns\te_1 + \te_2\tns\te_2,
\end{equation}
and where $\te_i := \omg_i\circ\wtbt$ is a differential $1$-form along $\wtpi$ for $i=1,2$.
By definition $\bg^\wtpi$ extends to the whole of $T^\wtpi X_0$ the restriction of the fiber-metric 
$(\bg_0|_{X_0\setminus Y_0})\circ\wtpi$ to $T^\wtpi X_0|_{\wtX\setminus\wtE}$.

\smallskip
From Theorem \ref{thm:main}, each point $\wta$ of $\wtX$ admits a neighbourhood $\wt{\cU}$ such that for $i=1,2$, 
there exists a $1$-form $\mu_i$ of $\Omg_{\wt{\cU}}^1$, only with singularities adapted to $\wtE$, such that 
\begin{center}
$\wtbt^*(\te_i\circ\Delta\sgm_1) = \cM_i \mu_i$ 
\end{center}
with $\cM_i$ a monomial in $\wtE$, and $\Delta\wtpi:T\wtX \to T^\wtpi X_0 \subset (TM_0)^\wtpi$.
Let $\chi_1,\chi_2$ be local regular sections $\wtX \to T^\wtpi X_0$, which are orthogonal for the 
fiber-metric $\bg^\wtpi$ on $T^\wtpi X_0$, and such that $\theta_i (\chi_j) = \delta_{i,j}$ 
for $i,j\in\{1,2\}$. Suppose that $\wtcU$ is small enough such that we can choose 
regular coordinates $(u,v)$, centered at $\wta$ and adapted to $\wtE$, i.e. $(E,\wta)\subset\{uv=0\}$.   

Let $Q$ be the matrix of the mapping $(\Delta\sgm_1)\circ D\wtbt$ in the basis $(\partial_u,\partial_v)$ and 
$(\chi_1,\chi_2)$. Let $\adj (Q)$ be the adjugate matrix of $Q$ so that $\adj(Q)\cdot Q = Q\cdot \adj(Q) = \psi\cM\cdot Id$,
where $\cM$ is a monomial in $\wtE$ and $\psi$ an analytic unit over $\wt{\cU}$. 
Note that $\psi\cM$ is the (oriented) volume of the image of $\wtpi$ and by Theorem \ref{thm:main}, 
we have $\cM = \cM_1 \cM_2 \cM_{1,2}$. 

We obtain two regular vector fields on $\wtX$, namely $\zeta_i:=\adj(Q) \chi_i$, for $i=1,2$. They may vanish 
only on $\wtE$. We deduce that $\cM_i \mu_i (\zeta_j) = \theta_i (\cM\chi_j) = \psi\cM\delta_{i,j}$ for $i,j=1,2$.
Writing $\mu_i = a_i \rd u + b_i \rd v$ and $\zeta_i = \alpha_i\partial_u + \beta_i\partial_v$
we observe that 
\begin{eqnarray}
a_i\alpha_j + b_i \beta_j & = & \psi\cM_k\cM_{1,2}\delta_{i,j} \; \mbox{\rm with } i\neq k \; \mbox{\rm and } i,j,k\in\{1,2\}\\
a_1b_2 - a_2b_1 & = & \psi\cM_{1,2}.
\end{eqnarray}
Since $\mu_1$ and $\mu_2$ may only vanish at isolated points, 
such that 
\begin{center}
$(\alpha_i,\beta_i) = f_j \cdot (b_j, -a_j)$ with $i\neq j$
\end{center}
In other words, in the basis above, the mapping 
$(\Delta\sgm_1)\circ D\wtbt$ over $\wt{\cU}$ writes as 
\begin{multline*}
(\Delta\sgm_1)\circ D\wtbt = (f_1 \mu_1,f_2 \mu_2): T_\ulb \wtX\ni\xi \to (f_1 \mu_1 (\xi),f_2 \mu_2 (\xi))
\\ 
= (\theta_1,\theta_2) ([(\Delta\sgm_1)\circ D\wtbt] \cdot\xi) \in T_\ulb^\wtpi X_0.
\end{multline*}
Note also that along the way, we have proved the following expected 
\begin{lemma}\label{lem:differential}
The $\cO_\wta$-module $(\Omg_{{\sgm_1},\wta}^1)\circ\Dlt\wtpi$ is generated by $\cM_1\mu_1$ and 
$\cM_2\mu_2$.
\end{lemma}
%
%
%
%
%
%
%
%
%
%
%
%
%
%
%
%
%
%
%
%
%
%
%
%
%
%
%
%
%
%
%
%
\subsection{Local normal form of differentials}\label{subsection:norm-form-diff}
Let $\wta$ be a point of $\wtE$. Let $(u,v)$ be local coordinates centered at $\wta$ adapted to $\wtE$, so 
that 
$$
\{u=0\} \subset (\wtE,\wta) \subset \{uv=0\} .
$$

\medskip
The $\cO_{\wtX}$-module $(\Omg_{M_0}^1)^\wtpi$ is locally free of rank $n$ and, 
by definition, $\wtpi^*\Omg_{M_0}^1 = (\Omg_{M_0}^1)^\wtpi\circ\Dlt\wtpi$.
We recall that $T^\wtpi X_0$ is a vector sub-bundle of $TM_0^\wtpi|_{\wtX}$. Let $\ula_0 := \wtpi (\wta)$ 
be the image of $\wta$. For a germ of differential form $\te\in \Omg_{M_0,\ula_0}^1$, 
the local section $\te^\wtpi$ nearby $\wta$ of $T^\wtpi X_0^\vee$ is defined as
the restriction $\te\circ\wtpi|_{T^\wtpi X_0}$. Let 
\begin{center}
$\Lambda_\wtpi := (\Omg_{M_0}^1)^\wtpi|_{T^\wtpi X_0}$ 
\end{center}
be the $\cO_{\wtX}$-sub-module of $\Omg_\wtpi^1$ generated by the restrictions to $T^\wtpi X_0$. 

\medskip\noindent
{\bf Claim 1:} $\wtpi^*(\Omg_{M_0}^1) = \Lambda_\wtpi\circ\Dlt\wtpi$.
\begin{proof}[of Claim 1.]
Any germ  at $\wta$ of a vector field $\xi$ along $\wtX$
induces the germ at $\wta$ of the local section $D\wtpi\cdot\xi: (\wtX,\wta) \to T^\wtpi X_0$. 
For $\te\in \Omg_{M_0,\ula_0}^1$, we get that $\wtpi^*\te$ belongs to $\Omg_{\wtX,\ula}^1$ and for every $\wtb$ nearby $\wta$,
the linear form $(\wtpi^*\te)(\wtb)$ is defined as 
\begin{center}
$T_\wtb \wtX \ni \xi \to (\te(\wtpi(\wtb)))( (D\wtpi)(\wtb)\cdot \xi)$, 
\end{center}
while the linear form $(\te^\wtpi\circ\Dlt\wtpi)(\wtb)$ is defined as 
\begin{center}
$T_\wtb \wtX \ni \xi \to (\te(\wtpi(\wtb))|_{(T^\wtpi X_0)_\wtb}) ((D\wtpi)(\wtb)\cdot \xi)$, 
\end{center}
so that they coincide since $(D\wtpi)(\wtb)\cdot\xi$ lies in $(T^{\wtpi} X_0)_\wtb$ which is contained in 
$T_{\wtpi(\wtb)}M_0$.
\end{proof}

\medskip\noindent
{\bf Claim 2:} $\Lambda_\wtpi = \Omg_\wtpi^1$.
\begin{proof}[of Claim 2.]
We just need to show that $\Omg_\wtpi^1\circ\Dlt\wtpi$ is a subset of $\wtpi^*(\Omg_{M_0}^1)$. 
Let $\theta_1,\theta_2$ as in Equation (\ref{eq:metric-param-2}) and let $\chi_1,\chi_2$ 
be the dual basis (for the fiber metric $\bg^\wtpi$). 
Let $\omg_1,\omg_2$ be two $1$-forms of $\Omg_{M_0,\ula_0}^1$ such that $\omg_i (\ula_0) (\chi_j(\wta)) = \delta_{i,j}$.
Thus the sections $\omg_1^\wtpi$ and $\omg_2^\wtpi$ are linearly independent nearby $\wta$, and
the claim is proved.
\end{proof}
Combining Claim 1 and Claim 2 with Lemma \ref{lem:differential} yields the following important
\begin{proposition}\label{prop:differential}
The $\cO_\wtX$-module $\wtpi^*\Omg_{M_0}^1$ is locally generated at $\wta$ by $\cM_1\mu_1$ and $\cM_2\mu_2$.
\end{proposition}
Let $(x,y,z_3,\ldots,z_n)$ be local regular coordinates centered at the image point $\ula_0 := \wtpi (\wta)$.
Obviously $\wtpi^*(\Omg_{M_0}^1)$ is $\cO_\wta$-generated by $\wtpi^*\rd x, \wtpi^*\rd y, 
\wtpi^*\rd z_3, \ldots, \wtpi^*\rd z_n$. Since it is of local rank $2$, we can assume that the 
coordinates at $\ula_0$ were such that it is generated 
by $\wtpi^*\rd x, \wtpi^*\rd y$. 
Proposition \ref{prop:differential} implies, up to a linear change in $x$ and $y$,
that nearby $\wta$ the following relations hold:
\begin{eqnarray}
\label{eq:diff-mu1}
Unit \cdot \cM_1 \mu_1 & = & \wtpi^*\rd x + A \wtpi^*\rd y \\
\label{eq:diff-mu2}
Unit \cdot \cM_2 \mu_2 & = & B\wtpi^*\rd x + \wtpi^*\rd y
\end{eqnarray}
for $A,B \in \bm_\wta\cO_\wta$. 

\bigskip\noindent
$\bullet$ Assume that $\wta$ is a regular point of $\wtE$. 

\medskip
We start with the following obvious
\begin{lemma}\label{lem:imp-norm-pair}
Since $\mu_1\wedge\mu_2 = Unit \cdot\cM(\rd u\wedge \rd v)$, if $\mu_i (\wta) =0$ and $\mu_j (\wta) \neq 0$  
when $(i,j) =(1,2)$ or $(i,j) =(2,1)$, then $\mu_j = \rd u + u(\ldots)\rd v$
and $\cM = u^{1+t}$ for some non-negative integer $t$. 
\end{lemma}
When $\wta$ is a regular point of $\wtE$, we find $\cM_1 = u^{r}$ and $\cM_2 = u^s$ with $s\geq r\geq 0$. 
Let us write $\mu_i = a_i \rd u + b_i \rd v$, and $x_w$ for $\dd_w (\wtpi^*x)$ and 
$y_w$ for $\dd_w (\wtpi^*y)$, where $w$ is either $u$ or $v$. From Equations (\ref{eq:diff-mu1}) and 
(\ref{eq:diff-mu2}) we find the following relations: 
\begin{eqnarray}
\label{eq:xu+Ayu}
x_u + A y_u & = & u^r \psi_1 a_1\\
\label{eq:Bxu+yu}
B x_u + y_u & = & u^r \psi_2 (u^{s-r}a_2) \\  
\label{eq:xv+Ayv}
x_v + A y_v & = & u^r \psi_1 b_1\\
\label{eq:Bxv+yv}
B x_v + y_v & = & u^r \psi_2 (u^{s-r}b_2) 
\end{eqnarray}
where each $\psi_i$ is a local unit.

We deduce that $x = a_0 + u^{1+r} X (u,v)$ and $y = b_0 + u^{1+r} Y(u,v)$ for constants $a_0,b_0$.
So we can write 
\begin{center}
$\begin{array}{rcl}
X & = & x_0 + vx_1(v) + ux_2 (u) + uv x_3(u,v) 
\\ 
\vspace{4pt}
Y & = & y_0 + vy_1(v) + uy_2 (u) + uv y_3(u,v) .
\end{array}
$
\end{center}
We are using this local description of the blowing-up mapping to obtain the following 
possible local forms.
\begin{proposition}\label{prop:norm-form-smooth}
Assume the point $\wta$ is a regular point of $\wtE$.

\smallskip\noindent
1) If $\mu_1 (\wta) \neq 0$, then we find $\mu_1 = \rd u + u(\cdots)\rd v$. 

\smallskip\noindent
2) Suppose $\mu_1 (\wta) =0$ and write 
\begin{center}
$\mu_1 = [v^k \phi (v) + uc_1(u,v)] \rd u + uD \rd v$, with $k=1$ and $\phi (0) \neq 0$ if $D(\wta) =0$. 
\end{center}
We are in one of the situations listed below:

i) Suppose $k=1$ and $D(\wta) \neq 0$. We can choose the local regular coordinates $(u,v)$ centered at $\wta$
and adapted to $\wtE$ such that $x = a_0 + u^{r+1}v$ and $y = b_0 \pm u^{r+1}$ 
and thus $r=s$ and $t=1$. Moreover we find out that $\mu_2 = \rd u + u (\ldots)\rd v$.

ii) If $k=1$ and $D(\wta) = 0$, then $r=s$ and $\mu_2 = \rd u + u(\cdots)\rd v$.

iii) If $D(\wta) \neq 0$ and $k\geq 2$, then the conclusion of point i) holds true. 
\end{proposition} 
\begin{proof}
1) Since $u^{r+1} (X_v + A Y_v) = u^r \psi_1 b_1$, we get $b_1 = u c_1$ for some $c_1\in \cO_\wta$. 

\medskip\noindent
2) Suppose that $\mu_1 (\wta)= 0$.

\smallskip 
i) Assume that $\mu_1 = u \rd v + (v\phi + uc_1)\rd u$ with $\phi$ a local analytic unit 
such that $-\phi (0) \notin \Q_{>0}$. Since $a_1(\wta) = 0$ we find that $x_0 =0$. 
From Equation (\ref{eq:xv+Ayv}), we find $X_v + AY_v = \psi_1$ with $A(\wta) =0$, and we see 
that $x_1(0) \neq 0$. Let $\brv := X(v,u)$. Thus $(u,v) \to (u, X(u,v))$ is a regular change 
of coordinates so that $x = a_0 + u^{r+1}\brv$.
Since $v = \brv z_1 (\brv) + u z_2(u,\brv)$, with $z_1 (0) \neq 0$, we deduce 
that $\mu_1 = Unit [u\rd \brv + (\brv \bar\phi (\brv) + u \bar{c}_1)\rd u]$ with $\bar\phi(0) \neq 0$.

\smallskip\noindent 
Suppose the coordinates $(u,v)$, centered at $\wta$ 
and adapted to $\wtE$, are also such that 
\begin{center}
$\mu_1 = u \rd v +(v\phi (v) +  uc_1)\rd u$, with $-\phi (0) \notin \Q_{\geq 0}$,
\end{center}
and $x = a_0 + u^{r+1}v$. Since $\psi_1\cdot u^r \mu_1 = \rd x + A \rd y$, we get  
\begin{center}
$\psi_1\cdot [u \rd v + (v\phi + uc_1)\rd u] = [(r+1)v\rd u + u \rd v] + A[((r+1)Y + uY_u)\rd u + uY_v \rd v]$  
\end{center}
with $A(\wta) =0$ and $\psi_1$ a local analytic unit.
Thus we find
\begin{center}
$
\begin{array}{rcl}
\psi_1 & = & 1 + u A Y_v
\\
\psi_1 (v\phi + uc_1) & = & u + A[(r+1)Y + uY_u]
\end{array}
$
\end{center}
We deduce that $y_0 \neq 0$, so that $Y$ is a local analytic unit, and thus $y = b_0 + u^{r+1} Y$.
Let $\ve$ be the sign of $y_0$. Let $\bru (u,v) = u (\ve Y)^{\frac{1}{r+1}}$. The change of coordinates $(u,v) \to (\bru,v)$ 
is regular, centered at $\wta$ and adapted to $\wtE$, and we have $y = b_0 + \ve \bru^{r+1}$.
Thus $x = a_0 + \zeta(\bru,v)\bru^{r+1}v$ for a local analytic unit $\zeta$. 
Thus taking $\brv := v \zeta$, we have found local coordinates centered at $\wta$ adapted to $\wtE$ and 
such that 
\begin{center}
$x = a_0 + \bru^{r+1}\brv$ and  $y = b_0 + \ve\bru^{r+1}$
\end{center}
so that $r=s$. This implies 
that $t=1$, which can only occur if $\mu_2(\wta) \neq 0$ (otherwise $t\geq 2$).
Since $r=s$ and $\mu_2(\wta) \neq 0$, we deduce from point 1) that $\mu_2 = Unit(\rd\bru + 
\bru \bar{c}_2 \rd \brv)$.

\smallskip
ii) Assume that $\mu_1 = (v + uc_1) \rd u + u D \rd v$ with $D(\wta) = 0$.
Let us write $A = vA_1(v) + u(\ldots)$. 
Equation (\ref{eq:xu+Ayu}) provides
\begin{equation}\label{eq:xu+Ayu-simple}
(1+r)[x_0 + v(x_1(v) + [y_0 +vy_1(v)] A_1(v)) + u(\ldots)] = [v\psi_1(0,v) + u(\ldots)]. 
\end{equation}
We deduce that 
\begin{center}
$x_0 =0$ and $(1+r)v(x_1(v) + y_0A_1(v)) = v\psi_1(0,v)$.
\end{center}
Thus $x_1 (v)+ y_0 A_1(v)$ is an analytic unit. Since $D(\wta) = A(\wta) =0$ and by Equation (\ref{eq:xv+Ayv})
we get
\begin{equation}
X_v + A Y_v = \psi_1\cdot D.
\end{equation}
We deduce that $x_1 (0) = 0$ so that $y_0 \neq 0$. Up to a change of coordinates as in i), 
we can assume that $Y = b_0 \pm u^{r+1}$. Equation (\ref{eq:Bxu+yu}) reads 
\begin{equation}
B[(r+1)X + uX_u] + (r+1)Y +u Y_u = \psi_2 u^{s-r} a_2,
\end{equation}
and provides $a_2(\wta) \neq 0$ and $r=s$. Tanks to this latter condition we are back in point 1) 
by permuting $\mu_1$ and $\mu_2$ so that $\mu_2 = \rd u + u(\ldots)\rd v$. 
 
\smallskip\noindent
iii) Suppose $\mu_1 = (v^{k+2}\phi(v) + uc_1)\rd u + u\rd v$ with $k\geq 0$ so that $D\equiv 1$.
From Equation (\ref{eq:xv+Ayv}), we deduce that $x_1(0) =1$. Adapting Equation (\ref{eq:xu+Ayu-simple})
to our situation we get 
\begin{center}
$(x_1 + y_0 A_1)(0) = (v^{k+1}\psi_1(0,v))(0)=0$, 
\end{center}
so that $y_0A_1(0) = - x_1 (0) = 1$. So we have $x_1(0) \neq 0, y_0 \neq 0$, thus we reach, 
after two changes of variables (one to change $u$ and the next one to change $v$, the same 
conclusion as i).
\end{proof}
\begin{remark}
An obvious, but unexpected, consequence of Proposition \ref{prop:norm-form-smooth} is that any regular point of 
$\wtE$ cannot be a simultaneous singular point of both (local) foliations $\cF_1$ and $\cF_2$. Moreover according 
to our notations, at any regular point $\wta$ of $\wtE$ we can assume that we always have $\mu_1 (\wta) \neq 0$.
\end{remark}
We now use these pairs of normal forms to obtain the following Hsiang \& Pati type result
\begin{proposition}\label{prop:diff-smooth}
Let $\wta$ be a regular point of $\wtE$. 

There exist local coordinates $(u,v)$ centered at $\wta$ and adapted to $\wtE = \{u=0\}$ 
such that 

\smallskip
1) If $\mu_2 (\ula) \neq 0$, then the module $\wtpi^*\Omega_{M_0}^1$ is locally 
generated at $\ula$ by $\rd (u^{r+1})$ and $\rd (u^{s+1+m}v)$ 
for a non-negative integer $m$. 

\smallskip
2) If $\mu_2 = a_2\rd u + u\rd v$ with $a_2\in\bm_\wta$, then $t=1$ and
the module $\wtpi^*\Omega_{M_0}^1$ is locally generated at $\wta$ 
by $\rd (u^{r+1})$ and $\rd (u^{s+1}v)$. 

\smallskip
3) If $\mu_2 = (v+uc_2) \rd u + u e_2\rd v$ with $e_2\in\bm_\wta$, 
then $t\geq 2$ and the module $\wtpi^*\Omega_{M_0}^1$ is locally generated at $\wta$ 
by $\rd (u^{r+1})$ and $\rd (u^{s+t}v)$. 
\end{proposition}
\begin{proof}
For simplicity let $\Te : = \wtpi^*\Omg_{M_0}^1$. We recall that 
$\mu_1\wedge\mu_2 = Unit\cdot u^t \rd u\wedge\rd v$. 

\smallskip\noindent
By Proposition \ref{prop:norm-form-smooth} we find $\mu_1 = \rd u + u c_1 \rd v$.
Equation (\ref{eq:xu+Ayu}) gives $x_0 \neq 0$.
Up to an adapted change of coordinates in $u$, we assume that $x = a_0 \pm u^{r+1}$.
Up to replacing $y$ by $y \pm y_0 x$, we assume that $y_0 = 0$. 

\medskip
1) Suppose $\mu_2 (\wta) \neq 0$.

If $t =0$, then we can assume that $\mu_2 = \rd v$. 
Equation (\ref{eq:Bxv+yv}) provides $Y_v = \psi_2 u^{s-r}$, so that up to changing 
$v$ by $Unit \cdot v$, we can assume $y = b_0 + u^{r+2}y_1(u) + u^{s+1}v$. So we get 
result in this case.

\smallskip
If $t\geq 1$ and $\mu_2 (\wta) \neq 0$ then we deduce $\mu_2 = \rd u +  (uc_1 + Unit\cdot u^t)\rd v$. 
And we still have $\mu_2 = u\rd v+ a_2\rd u$ but only $\mu_1 = \rd u + u(\ldots)\rd v$. 
Writing $Y = uy_1(u) + u^pvZ(u,v)$ for a non negative integer $p$ and with $Z$ such that $Z(0,v) \not\equiv 0$,
we deduce from Equations(\ref{eq:xv+Ayv}) and (\ref{eq:Bxv+yv}) 
\begin{center}
$A Y_v = u^p A (Z+vZ_v) = \psi_1 c_1$ and $Y_v = \psi_2 u^{s-r}(c_1 + u^{t-1}\psi_3)$   
\end{center}
where $\psi_3$ is a local unit. Since $A(\wta) = 0$, we deduce 
$Z+vZ_v = Unit \cdot u^{s+t-r-1-p}$ and therefore $p = s-r+t-1$. 
Thus up to changing $v$ into $Unit \cdot v$, we can assume $y = b_0 + u^{r+2}y_1(u) + u^{s+t}v$. So we 
also get the desired result in this case.

\medskip
2) Assume $\mu_2 = u\rd v + a_2 \rd u$ with $a_2 \in \bm_\wta$. Thus $t = 1$ and from Equation (\ref{eq:Bxv+yv})
we get $Y_v = \psi_2 u^{r-s}$. So that after a change of coordinates in $v$
we can assume that $y = b_0 + u^{r+2}y_1(u) + u^{s+1}v$. And we find the announced result. 

\medskip
3) Suppose $\mu_2 = (v+uc_2) \rd u + u e_2\rd v$ with $e_2 \in \bm_\wta$. 
Thus we deduce $t\geq 2$.
Since $\mu_1 = \rd u + u c_1\rd v$ we deduce that that $\mu_2 = (\cdots)\mu_1 + u^t \rd v$. 
The module $\Te$ is generated by $u^r\mu_1$ and $u^{s+t}\rd v$.
We will use the following 
\begin{lemma}\label{lem:change-mu1}
Suppose $\mu_1 = \rd u + u^p C\rd v$, with $p\geq 1$ so that $C(0,v)\not\equiv 0$.
There exists a change of coordinates $(w,v) \to (w + w^p \alpha(v),v)$ such that $\alpha (0) =0$ and
\begin{center}
$\mu_1 = Unit [\rd w + w^{p+1}(\cdots) \rd v]$. 
\end{center}

\end{lemma}
\begin{proof}
Let $C = c_0(v) + u C_1 (u,v)$, so that $c_0 (v)$ is not identically $0$.
Let $u = w + w^p \alpha (v)$. 
Since $\rd u = [1+ pw^{p-1}\alpha(v)]\rd w + w^p\alpha'(v)\rd v$. 
We get 
\begin{center}
$\mu_1 = [1 + pw^{p-1}\alpha(v)]\rd w  + w^p(\alpha'(v) + c_0(v))\rd v + w^{p+1}\beta (v,w)\rd v$
\end{center}
with $\beta \in \cO_\wta$. Taking $\alpha$ the primitive of $-c_0$ vanishing at $v=0$ provides the result.   
\end{proof}
\begin{remark}\label{rem:formal-coord}
The equation we solve in $W:= 1 + w^{p-1}\alpha(v)$ in the proof of the Lemma admits a formal solution (that 
is a solution in the real formal power series in two variables), so that up to a formal change of 
coordinates $u = \bru W(\bru,v)$ 
for $W$ a formal power series and a unit, we would find $\mu_1 = \rd \bru$.
\end{remark}
Thanks to the Lemma, used finitely many times, we deduce (despite these changes of coordinates) 
that $\Te$ is generated by $u^r\rd u$ and $u^{s+t}\rd v$ which is the result.
\end{proof}
Using Hsiang \& Pati proof (and caring with the fact that $-1$ has no real square root), 
with a few elementary computations, we deduce from Proposition \ref{prop:diff-smooth}
the existence of Hsiang \& Pati coordinates:
\begin{corollary}[see also \cite{BBGM}]\label{cor:real-HP-smooth}
There exist adapted coordinates $(u,v)$ at the regular point $\wta$ such that the resolution mapping $\wtpi$
locally writes 
\begin{center}
$(u,v)  \to (x,y;z) = \wtpi(\wta) + (\pm u^{k+1},u^{k+1}f(u) + u^{l+1}v; u^{k+1}z(u) + u^{l+1}vZ(u,v)) \hfill$

\smallskip
$\hfill \in \R\times\R\times\R^{N-2}$
\end{center}
for non-negative integers $k,l$, for germs $f \in \R\{u\}$, $z \in \R\{u\}^{N-2}$ and 
$Z$ a regular map germ $(X,\ula) \to \R^{N-2}$.
\end{corollary}

\medskip\noindent
$\bullet$ Assume now that $\wta$ is a corner point of $\wtE$. 

\medskip
If $\wta$ is a regular point of $\mu_i$, for $i=1$ or $2$, then we can write $\mu_i = \rd z + z(\cdots)\rd w$
where $(w,z) = (u,v)$ or $(v,u)$. In this case the logarithmic local generator of $\cF_i$ 
writes 
\begin{center}
$\mu_i^{\log} = z^{-1} \mu_i = \rdlg z + w (\cdots) \rdlg w$.
\end{center}

\smallskip
If $\wta$ is a singular point of $\mu_i$ (thus adapted to $\wtE$), for $i=1$ or $2$, then $\mu_i = w\rd z + z (\cdots) \rd w$
and where $(w,z) = (u,v)$ or $(v,u)$. 
The logarithmic local generator of $\cF_i$ writes in this case 
as 
\begin{center}
$\mu_i^{\log} = w^{-1}z^{-1} \mu_i = \rdlg z + (\cdots) \rdlg w$.
\end{center}

\medskip 
By Proposition \ref{prop:differential}, the sub-module $\wtpi^*\Omg_{M_0,\wtpi (\wta)}^1$ 
of $\Omg_{\wtX,\wta}^1$ is generated over $\cO_\wta$ by $\cM_1\mu_1$ and $\cM_2\mu_2$. 
Thus $\wtpi^*\Omg_{M_0}$, as a $\cO_\wtX$- sub-module of $\Omg_\wtX^1 (\log \wtE)$
is locally generated at $\wta$ by $\cM_1^{\log}\mu_1^{\log}$ and $\cM_2^{\log}\mu_2^{\log}$.

\smallskip
We know that 
\begin{center}
$\cM_1^{\log} = u^mv^n$ and $\cM_2^{\log} = u^pv^q$ with $\mu_1^{\log}\wedge\mu_2^{\log} = 
Unit\cdot u^rv^s (\rdlg u \wedge \rdlg v)$. 
\end{center}
We can assume by Theorem \ref{thm:main} that 
$p\geq m\geq 0$, $q\geq n\geq 0$ and $m+n\geq 1$. When $\mu_1$ or $\mu_2$ vanishes at $\wta$, we deduce $\max (r,s)\geq 1$.
Since the local coordinates $(u,v)$ are centered at $\wta$ and adapted to $\wtE$, for  $i=1,2$ 
we can write $\mu_i^{\log} = a_i \rdlg u + b_i \rdlg v$ (and $a_i$ or $b_i$ is a local
unit), so that, up to permuting $u$ and $v$ we can always assume that $a_1(\wta) \neq 0$, thus $m\geq 1$.

\smallskip
Using again Equations (\ref{eq:diff-mu1}) and (\ref{eq:diff-mu2}) we obtain the following 
relations:
\begin{eqnarray}
\label{eq:xu+Ayu-log}
ux_u + uA y_u & = & u^mv^n \psi_1 a_1\\
\label{eq:Bxu+yu-log}
uBx_u + uy_u & = & u^mv^n\psi_2 (u^{p-m}v^{q-n}a_2) \\  
\label{eq:xv+Ayv-log}
vx_v + vA y_v & = & u^mv^n \psi_1 b_1\\
\label{eq:Bxv+yv-log}
vB x_v + vy_v & = & u^mv^n \psi_2 (u^{p-m}v^{q-n}b_2) 
\end{eqnarray}
We deduce, up to changing $v$ into $Unit\cdot v$, that 
\begin{center}
$x = X_0 (v) \pm  u^mv^n\;$ and $\;y = Y_0(v)+ u^mv^n Y(u,v)$.
\end{center}
Since $m\geq 1$, Equations (\ref{eq:xv+Ayv-log}) and (\ref{eq:Bxv+yv-log}) imply
that $X_0 = a_0$ and $Y_0 = b_0$ are real numbers. 
So we can write 
\begin{center}
$
Y  =  y_0 + vy_1(v) + uy_2 (u) + uv y_3(u,v) .
$
\end{center}
From Equation (\ref{eq:xv+Ayv-log}), we deduce that necessarily $b_1 (0) \neq 0$,
thus we have deduced 
\begin{lemma}\label{lem:mu1-corner}
For any local coordinates $(u,v)$ centered at the corner point $\wta$ of $\wtE$ and 
adapted to $\wtE$ we find $\mu_1 = v\rd u + Unit\cdot u\rd v $.
\end{lemma}
\begin{proposition}\label{prop:diff-corner}
There exist local coordinates $(u,v)$ centered at $\wta$ adapted to $\wtE$ such that
$\wtpi^*\Omg_{M_0}^1$ is locally generated at $\wta$ as an $\cO_\wtX$-sub-module of $\Omg_\wtX^1 (\log \wtE)$
by 
\begin{center}
$u^mv^n\rdlg (u^mv^n)$ and $u^{r+p}v^{s+q} \rdlg (u^{r+p}v^{s+q})$. 
\end{center}
Equivalently $\wtpi^*\Omg_{M_0}^1$ is $\cO_\wta$-generated, as a submodule of $\Omg_{\wtX,\wta}^1$ 
nearby $\wta$ by $\rd (u^mv^n)$ and $\rd (u^{r+p}v^{s+q})$. Therefore the plane vectors $(m,n)$
and $(p+r,q+s)$ are linearly independent.
\end{proposition}
\begin{proof}
We already have that $x=a_0 \pm u^mv^n$, the sign "$\pm$" may be "$-$" only if $m,n$ are both even.

\smallskip
We can write $y  - b_0 = u^mv^n[f(u,v) + z(u,v)]$ for regular germ $f,z \in \cO_\wta$ such that 
each monomial $u^kv^l$ appearing in $f$ is such that $ml - kn =0$, and each monomial $u^kv^l$ appearing in $z$ 
is such that $ml- kn \neq 0$. Thus 
\begin{center}
$\rd x \wedge \rd y = u^mv^n \rd x \wedge \rd z = 
Unit\cdot u^{r+p+m-1}v^{s+q+n-1}\rd u \wedge \rd v$.
\end{center}
Let $u^kv^l$ be a monomial of $z$, thus $\rd x \wedge \rd (u^kv^l) = (ml-nk) u^{k+m-1}v^{l+n-1} \rd u \wedge \rd v$.
Necessarily we deduce that $z = u^{r+p}v^{s+q}\alpha$ for a local analytic unit $\alpha$.
Thus the plane vectors $(m,n)$ and $(p+r,q+s)$ are 
linearly independent. We are looking, 
if possible, for a change of local coordinates of the form $u = \bru U$ and $v = \brv V$ for local units 
$U,V$ such that $\bru^m\brv^n = \pm u^mv^n$ and $u^{r+p}v^{s+q}\alpha = \pm \bru^{r+p}\brv^{s+q}$. 
Let $\ve$ be the sign of $\alpha (0)$.
So we need $U^mV^n = 1$ and $U^{r+p}V^{s+q} = \ve \alpha$ knowing that $m(s+q) - n (r+p) \neq 0$, 
this is equivalent to $V^{(s+q)n - (r+p)m} = (\ve \alpha)^m$, which can be solved. 
Thus we can re-write $x = a_0 \pm \bru^m\brv^n$ and $y = a_0 + \bru^m\brv^n[h(\bru,\brv) \pm 
\bru^{r+p}\brv^{s+q})]$ where $h$ has only monomials $\bru^k\brv^l$ such that $ml-kn=0$.
These coordinates satisfy the announced result.
\end{proof}
As in the case of a regular point, using Hsiang \& Pati proof, with a few elementary computations, 
we deduce from Proposition \ref{prop:diff-corner} the existence of Hsiang \& Pati coordinates:
\begin{corollary}[see also \cite{BBGM}]\label{cor:real-HP-corner}
There exist adapted coordinates $(u,v)$ at the corner point $\wta$ such that the resolution mapping $\wtpi$
locally writes 
\begin{center}
$(u,v)  \to (x,y;z) = \wtpi (\wta)+(\pm u^mv^n, u^mv^nf(u,v) \pm u^kv^l;  u^mv^nz(u,v)+u^kv^l Z(u,v)) \hfill$

\smallskip
$\hfill \in \R\times\R\times\R^{n-2}$
\end{center}
for non-negative integers $k \geq m$ and $l\geq n$ such that $nk -lm \neq 0$ and germs $f \in \cO_\ula$ 
and $z,Z\in \cO_\ula^{N-2}$ such that $\rd f \wedge \rd(u^mv^n) = \rd z \wedge \rd(u^mv^n) \equiv 0$.
\end{corollary}
%
%
%
%
%
%
%
%
%
%
%
%
%
%
%
%
%
%
%
%
%
%
%
%
%
%
%
%
%
%
%
%
%
%
%
\subsection{Local normal form for the induced metric}\label{subsection:norm-form-metric}
$ $

\smallskip
Following on the material presented in the previous subsection, we will give local normal forms of the
metric $\bg_0|_{X_0}$ pulled-back onto the resolved surface $\wtX$ at any regular point $\wta$ of the exceptional 
divisor $\wtE$ and add a few precisions at corner points.

\medskip 
From Theorem \ref{thm:main}, Proposition \ref{prop:differential}, Proposition \ref{prop:diff-smooth}
and Proposition \ref{prop:diff-corner} there exist 
\begin{center}
$U:=\cM_1^{\log}$ and $T:=\cM_2^{\log}$ and $V:=T \cdot \cM_{1,2}^{\log}$,
\end{center} 
ordered monomials in $\wtE$ with $T = U \cdot (\cdots)$ and $V = T \cdot (\cdots)$ such that 
the pulled-back metric on $\wtX$ writes nearby $\wta$
\begin{eqnarray*}
\wtpi^*\bg_0|_{X_0} & = & \lambda_1 (\cM_1\mu_1)\tns (\cM_1\mu_1)  + \lambda_2(\cM_2\mu_2)\tns (\cM_2\mu_2) 
\end{eqnarray*}
for positive analytic units $\lambda_1,\lambda_2$. Moreover we know that
$\{U \mu_1^{\log},T\mu_2^{\log}\}$ and $\{U \rdlg U,V \rdlg V\}$ are both 
$\cO_\wta$-basis of the sub-module $\wtpi^*\Omg_{M_0}^1$ of $\Omg_{\wtX}^1(\log E)$ nearby $\wta$.
Thus we can write,
\begin{eqnarray}\label{eq:corners-log-1}
U \mu_1^{\log}& = & C_1U\rdlg U + D_1 V\rdlg V \\
\label{eq:corners-log-2}
T \mu_2^{\log} & = & C_2U\rdlg U + D_2 V \rdlg V
\end{eqnarray}
with $C_1D_2 - C_2D_1 = Unit$. Let us write 
\begin{center}
$\eta_1 := \rdlg U\;\;$ and $\;\;\eta_2 := \rdlg V$. 
\end{center}
If $T \neq U$ in Equation (\ref{eq:corners-log-2}), then $C_1D_2 = Unit$. 
Thus we can write  
\begin{center}
$\wtpi^*\bg_0|_{X_0} = (\lambda_1C_1^2 + \lambda_2C_2^2) \eta_1\tns\eta_1 + 
(\lambda_1D_1^2 + \lambda_2D_2^2) \eta_2\tns\eta_2 + \hfill$

\smallskip
$\hfill (\lambda_1C_1D_1 + \lambda_2C_2D_2) (\eta_1\tns\eta_2+\eta_2\tns\eta_1)$
\end{center}
for positive analytic units $\lambda_1,\lambda_2$.
Since $C_1D_2 - C_2D_1$ is a unit, as a quadratic form in the "variables" $\eta_1,\eta_2$,
the pulled-back metric $\wtpi^*\kp$ is positive definite nearby $\wta$, thus we
deduce the following looked for and expected 
\begin{proposition}\label{prop:quasi-iso-metric}
At the point $\wta$ of $\wtE$ the pulled-back metric $\wtpi^*\bg_0|_{X_0}$ is locally quasi-isometric to 
the following metric:
\begin{eqnarray*}
U^2\eta_1\tns\eta_1 + V^2\eta_2\tns\eta_2 & = &\rd U\tns\rd U + \rd V\tns\rd V \\
& = & \rd \cM_1^{\log}\tns\rd \cM_1^{\log} + \rd (\cM_2^{\log}\cM_{1,2}^{\log})\tns\rd(\cM_2^{\log} \cM_{1,2}^{\log}) .
\end{eqnarray*}
\end{proposition}

\medskip\noindent
$\bullet$ When $\wta$ is a regular point of $\wtE$, we can be a little more specific. We recall that 
by Lemma \ref{lem:change-mu1}, we can assume that for any given $p\geq s-r+t+2$, the coordinates 
$(u,v)$ are such that $\mu_1 = \rd u + u^pc_1 \rd v$.
As a consequence of Proposition \ref{prop:norm-form-smooth} and 
Proposition \ref{prop:diff-smooth} and of elementary computations we find the following 
normal forms:
\begin{proposition}\label{prop:metric-smooth}
Let $\wta$ be a regular point of $\wtE$ we obtain. For each integer number $\rho\geq 2s + 2t + 1$, there
exists an local adapted coordinates $(u,v)$ such that 

\smallskip
\begin{center}
$\wtpi^*\bg_0 = \lambda_1 u^{2r}\rd u\tns\rd u + \lambda_2 u^{2s+2t}\rd v\tns\rd v + u^{\rho}(\cdots) 
(\rd u\tns\rd v+\rd v\tns\rd u)$
\end{center}
for positive analytic units $\lambda_1$, $\lambda_2$.
\end{proposition}
\begin{proof}
Let $\rho$ be given. We can already assume that $p$ is chosen so that $2r + p \geq \rho$. 
We write $\mu_2$ as $B\mu_1 + u^t\rd v$ for $B\in \bm_\wta$ and $t \geq 1$.
We an write 
\begin{center}
$\wtpi^*\bg_0 = [Unit^2 \cdot u^{2r} + Unit^2 \cdot u^{2s+2t} B^2]\mu_1\tns\mu_1 + Unit^2 \cdot u^{2s+t} B (\mu_1 \tns\rd v + 
\rd v\tns\mu_1) \hfill$

\smallskip
$\hfill + Unit^2 \cdot u^{2s+2t} \rd v\tns\rd v$ .
\end{center}
Let us consider a change of variable of the form $u = w (1 + w^q A(v))$ with $q$ a positive integer. 
Then we deduce that 
\begin{eqnarray*}
\rd u & = & [1+ (q+1)w^q A] \rd w + w^{q+1} A'\rd v \\
\rd u \tns\rd v & = & [1+ (q+1)w^q A] \rd w \tns\rd v + w^q A'\rd v\tns\rd v   \\
\rd u \tns \rd u & = & [1+ (q+1)w^q A]^2 \rd w\tns\rd w + w^{q+1}[1 + (q+1)w^q A] A'(\rd w \tns\rd v + \rd v \tns\rd w) 
\\
& & + w^{2q + 2} (A')^2 \rd v\tns\rd v
\end{eqnarray*}
Observe that $\mu_1\tns\mu_1 = (\rd u)^2 + u^p(\cdots)$.
In the new coordinates we find 
\begin{center}
$\wtpi^*\bg_0|_{X_0} = Unit^2 \cdot w^{2r} \rd w\tns\rd w +  w^r C (\rd v\tns\rd w + \rd w\tns\rd v) + w^{2r} (\cdots) 
\rd v\tns\rd v$
\end{center}
where 
\begin{center}
$w^rC = (\lambda_1 w^{2r+q+1}[1 + (q+1)w^qA]^{2r+1}A' + \lambda_2 w^{2s+t}[1 + (q+1)w^qA]^{2s+t}[B + w^t(\cdots)]$
\end{center}
Since $B(0,v)= v^l b_0(v)$ for a positive integer $l$ and a local analytic unit $b_0(v)$, 
taking $q = 2s -2r +t -1$,
we find $A(v) = v^{l+1} a_0(v)$ for a local analytic unit $a_0$, resolving a differential equation 
in $v$ of the form $A' \psi (A) = v^l f(v)$ where $\psi$ and $f$ are local analytic unit in $v$, 
such that $w^rC = w^{2s +t +1}(\cdots)$.
The metric then writes 
\begin{center}
$\wtpi^*\bg_0|_{X_0} = Unit^2 \cdot w^{2r}\rd w\tns\rd w + Unit^2 u^{2s+2}(u^{t-1} \rd v+ G \rd w)\tns(u^{t-1} \rd v+ G \rd w)$ 
with $\; G \in \bm_{\wta}$.
\end{center}
Up to factoring further powers of $u$ from $G$, we may assume that $G(0,v) \neq 0$, which is the worse case scenario.
Replacing $\mu_1$ by $\rd w$ and $\mu_2$ be $H\rd u + u^{t-1}\rd v$, we check that 
after at most $t-1$ consecutive changes of coordinates in the \em exceptional \em variable $u_{exc}$, of the form 
\begin{center}
$u_{exc,old} := u_{exc,new} [1 + u_{exc,new}^{q_{new}} A_{new}(v)]$, 
\end{center} 
we find adapted coordinates $(x,v)$, with $(\wtE,\wta) = \{x=0\}$, such that   
\begin{center}
$\wtpi^*\bg_0|_{X_0} = Unit^2 \cdot u^{2r} \rd u\tns\rd u + Unit^2 \cdot u^{2s+2t}\rd v\tns\rd v + 
x^{2s+2t+1} (\cdots)(\rd u\tns\rd v + \rd v\tns\rd u)$.
\end{center}
From here, finitely many (iterated) changes of variables of the form $v = y + x^m J(y)$, for a positive integer
$m$ and regular function germ $J$ to find, will provide the 
announced result.
\end{proof}

\smallskip\noindent
$\bullet$ When $\wta$ is a corner point of $\wtE$, we know that the form $\mu_1$ (attached to the "smallest" 
monomial in $\wtE$, writes 
\begin{center}
$\mu_1 = uv[\rdlg u + (\lambda + G) \rdlg v]$, with $-\lambda\notin\Q_{\geq 0}$ and $G \in \bm_\wta$
\end{center}
for local adapted 
coordinates $(u,v)$ at $\wta$.
Since our result is general and we do not have explicit equations of the foliations we are dealing with,
the desingularization of the foliation locally given by $\mu_1$ will tell us very little about $\lambda$. 
Remarkably the very special context we are working in gives the value of $\lambda$, for free, namely 
\begin{corollary}
Since $U\mu_1^{\log} = U \rdlg U +  H V \rdlg V$ with $H \in \cO_\wta$, and $U = u^mv^n$,
for positive integer numbers $m,n$, we find that $\lambda = \frac{n}{m}$, in other words
we can choose $\mu_1$ such that 
\begin{center}
$\mu_1^{\log} = \rdlg U + F \rdlg v\, , \;$ with $\;F \in \bm_\wta.$
\end{center}
\end{corollary}
\begin{proof}
Let $\xi := nu\dd_u - mv\dd_v$ so that $\rdlg U (\xi) \equiv 0$. 
Using Proposition \ref{prop:quasi-iso-metric} and evaluating both quasi-isometric 
metrics along the vector field $\xi$ provides the result.
\end{proof}
%
%
%
%
%
%
%
%
%
%
%
%
%
%
%
%
%
%
%
%
%
%
%
%
%
%
%
%
%
%
%
%
%
%
%
%
%
%
%

%
%
\end{document}